\documentclass[reqno]{amsart}
\usepackage{amssymb,mathrsfs,graphicx,subfigure,enumerate}
\usepackage{amsmath,amsfonts,amssymb,amscd,amsthm,bbm}
\usepackage{graphicx,colortbl}
\usepackage{extpfeil}

\def\bU {\mathbf{U}}

\def\fS {\mathfrak{S}}

\def\fsu {\mathfrak{su}}

\def\cA {\mathcal{A}}

\def\cD {\mathcal{D}}
\def\cE {\mathcal{E}}
\def\cF {\mathcal{F}}

\def\cM {\mathcal{M}}

\def\cP {\mathcal{P}}

\def\cT {\mathcal{T}}
\def\cV {\mathcal{V}}

\def\a {{\alpha}}

\def\g {{\gamma}}

\def\de {{\delta}}
\def\eps {{\epsilon}}

\def\ka {{\kappa}}
\def\l {{\lambda}}
\def\L {{\Lambda}}
\def\si {{\sigma}}
\def\Si {{\Sigma}}

\def\Om {{\Omega}}

\def\d {{\partial}}

\def\Dlt {{\Delta}}

\def\rstr {{\big |}}

\def\la {\langle}
\def\ra {\rangle}

\def \lA {\big\langle \! \! \big\langle}
\def \rA {\big\rangle \! \! \big\rangle}

\def\wto {{\rightharpoonup}}

\newcommand{\Div}{\operatorname{div}}

\newcommand{\Diam}{\operatorname{diam}}

\newcommand{\Supp}{\operatorname{supp}}

\newcommand{\Tr}{\operatorname{trace}}

\newcommand{\Dist}{\operatorname{dist}}

\newcommand{\Lip}{\operatorname{Lip}}

\newcommand{\MK}{\operatorname{dist_{MK,1}}}
\newcommand{\MKd}{\operatorname{dist_{MK,2}}}

\newcommand{\bbr}{\mathbb R}
\newcommand{\bbc}{\mathbb C}

\newcommand{\bbs}{\mathbb S}

\def\wto {{\rightharpoonup}}

\newcommand{\ba}{\begin{aligned}}
\newcommand{\ea}{\end{aligned}}

\newcommand{\be}{\begin{equation}}
\newcommand{\ee}{\end{equation}}

\newtheorem{Thm}{Theorem}[section]
\newtheorem{Lem}{Lemma}[section]

\newtheorem{Prop}{Proposition}[section]

\newtheorem{Def}{Definition}[section]

\numberwithin{equation}{section}

\begin{document}

\title[Mean-Field Limit of Lohe Matrix Model]{A Mean-Field Limit of the Lohe Matrix Model and Emergent Dynamics}

\author[F. Golse]{Fran\c cois Golse}
\address[F.G.]{Ecole polytechnique, CMLS, 91128 Palaiseau Cedex, France}
\email{francois.golse@polytechnique.edu}

\author[S.-Y. Ha]{Seung-Yeal Ha}
\address[S.-Y. H]{Seoul National University, Dept. of Math. Sciences and Research Institute of Mathematics, Seoul 08826, Korea (Republic of) and Korea Institute for Advanced Studies, School of Mathematics, Seoul 130-722, Korea (Republic of)}
\email{syha@snu.ac.kr }

\thanks{\textbf{Acknowledgment.} The work of S.-Y. Ha is partially supported by National Research Foundation of Korea Grant (NRF-2017R1A2B2001864) funded by the Korea Government. This work was started, when the second author stayed 
National Center for Theoretical Sciences-Mathematics Division in National Taiwan University in the fall of 2015}

\begin{abstract}
The Lohe matrix model is a continuous-time dynamical system describing the collective dynamics of group elements in the unitary group manifold, and it has been introduced as a toy model of a non-abelian generalization of the Kuramoto phase 
model. In  the absence of couplings, it reduces to the finite-dimensional decoupled free Schr\"{o}dinger equations with constant Hamiltonians. In this paper, we study a rigorous mean-field limit of the Lohe matrix model which results in a Vlasov
type equation for the probability density function on the corresponding phase space. We also provide two different settings for the emergent synchronous dynamics of the kinetic Lohe equation in terms of the initial data and the coupling strength. 
\end{abstract}



\date{\today}

\maketitle



\section{Introduction} \label{sec:1}

Synchronization of weakly coupled oscillators is ubiquitous in our biological, engineering and physical complex systems, e.g., cardiac pacemaker cells \cite{Pe}, biological clocks in the brain \cite{Wi1}, Josephson junction arrays \cite{A-B, P-R} 
and flashing of fireflies \cite{B-B} etc (see \cite{A-B, D-B1, H-K-P-Z} for a detailed survey). It was first reported by Huygens with the two pendulum clocks hanging on the same bar on the wall in the middle of seventeen century (see \cite{P-R} 
for a brief history of synchronization). However, its rigorous and systematic study has begun by  A. Winfree \cite{Wi2} and Y. Kuramoto \cite{Ku1, Ku2} in only half century ago. Recently, the research on the collective dynamics of complex systems
 has received lots of attention in various scientific disciplines \cite{A-B, B-B, D-B1, Pe, P-R, St, Wi2, Wi1} mainly due to their possible engineering applications in sensor networks, control of robots and unmanned aerial vehicles etc. After Winfree 
 and Kuramoto's seminal works, many phenomenological models have been proposed. Among them, the Kuramoto model is a prototype one for classical phase coupled oscillators, and it has been extensively studied in literature (see for instance\cite{C-H-J-K, C-S, D-X, D-B1, D-B0, D-B, H-K-R, M-S1, M-S2, M-S3,V-M1, V-M2}). On the other hand, because emerging applications in control theory of quantum systems, quantum synchronization has been discussed in physics communities 
and several phenomenological master equations were proposed in \cite{A-E, H-L-G, G-C, G-G, Lo-1, Lo-2, X-T, Z-S} to describe quantum synchronization. 

Our main interest in this paper lies on the Lohe matrix model \cite{Lo-1, Lo-2} which is a continuous dynamical system on the unitary group. This model was proposed as a non-abelian generalization of the Kuramoto  phase model in \cite{Lo-1},
and it was further generalized in \cite{De,H-K-R2}.  
Next, we briefly describe the Lohe matrix model. Consider a network \cite{Ki, X-T, Z-S} consisting of $N$ nodes and edges connecting all possible pairs of nodes.  Each node can be viewed as a component of a physical system interacting via 
edges. For instance, atoms at nodes can have effect on spin-spin interactions generated by a single photon pulse traveling along the channels (see \cite{Ki} for a detailed description).  Let $U_j$ and $U_j^{*}$ be a $d \times d$ unitary matrix 
and its Hermitian conjugate, and $H_j$ a $d \times d$ Hermitian matrix whose eigenvalues are the natural frequencies of the $i$-th Lohe oscillator. Henceforth, we denote by $\bU(d)$ the group of $d \times d$ unitary matrices. In this situation, 
the Lohe matrix model reads as follows.
\begin{equation} \label{Lohe}
{\mathrm i} \dot{U}_j U_j^{*} = H_j + \frac{{\mathrm i}\kappa}{2N}\sum_{k = 1}^{N} \left(U_k U_j^{*} - U_j U_k^{*} \right), \quad j=1,\cdots,N,
\end{equation}
where $\kappa$ is a nonnegative coupling strength and ${\mathrm i} = \sqrt{-1}$. In the zero coupling case $\kappa = 0$, the model \eqref{Lohe} is a system of $N$ uncoupled free Schr\"{o}dinger equations with constant Hamiltonians. In the 
case of dimension $d=1$, system \eqref{Lohe} can also be reduced to the usual (abelian) Kuramoto model \cite{Ku1, Ku2}. In the case of dimension $d=2$, the Lohe model is equivalent to a ``consensus swarming model'' \cite{Ol} on the
sphere $\bbs^3$. See Section 2.2 for more details on these special cases.

Since the right hand side of \eqref{Lohe} is a self-adjoint matrix, say $S_j$, one has
\[
\frac{d}{dt}(U_jU^*_j)=\dot{U}_j U_j^*+U_j\dot U_j^*=-iS_j+iS^*_j=0
\]
for $j=1,\dots,N$. Since the system \eqref{Lohe} is autonomous, and $S_j$ is quadratic and therefore of class $C^1$ in its arguments $U_1,\ldots,U_N$, there exists a unique solution $t\mapsto(U_1(t),\dots,U_N(t))$ to the Cauchy problem 
for \eqref{Lohe} for any $N$-tuple of initial data $(U_1^0,\ldots,U_N^0)\in\bU(d)^N$ by the application of the standard Cauchy-Lipschitz theory. The last identity shows that $(U_1(t),\dots,U_N(t))\in\bU(d)^N$, and since $\bU(d)$ is compact, 
the usual continuation argument implies that the maximal solution $t\mapsto(U_1(t),\dots,U_N(t))$ is defined for all $t\in\bbr$.

Thus, one of  interesting issues for \eqref{Lohe} is the asymptotic dynamic behaviors such as  the formation of phase-locked states and its relaxation dynamics.  As discussed in \cite{A-B, C-H-J-K, C-S, D-B1},  the existence and stability of 
phase-locked states (a kind of relative equilibria) are important subjects in the nonlinear dynamics of the Kuramoto model. Henceforth we identify $U_j(t)$ as a state of the $j$-th Lohe oscillator. 

When the number of Lohe oscillators is large, it is virtually impossible to integrate the finite dimensional system \eqref{Lohe} directly. For that reason, we have addressed the following questions in the present paper: 
\begin{quote}
\begin{itemize}
\item
Can one rigorously derive a mean-field, kinetic type equation from the Lohe matrix model in the large $N$ limit?
\vspace{0.2cm}
\item
If so, under what conditions do solutions to the kinetic mean-field equation exhibit some type of synchronization?
\end{itemize}
\end{quote}

\smallskip
After reviewing the notation used below (most of which is fairly classical), we shall conclude this introduction with a short presentation of our main results, and with a detailed outline of the paper.


\subsection*{Notation.}


Let $M_d(\bbc)$ be the set of $d\times d$ square matrices with complex entries, and let $I_d$ be the identity matrix in $M_d(\bbc)$. As mentioned above, the (Lie) group of unitary matrices in $M_d(\bbc)$ is denoted by $\bU(d)$, i.e.
\[
\bU(d):=\{U\in M_d(\bbc)\hbox{ s.t. }UU^*=U^*U=I_d\}.
\]
The Lie algebra of $\bU(d)$ is
\[
\fsu(d):=\{A\in M_d(\bbc)\hbox{ s.t. }A+A^*=0\}.
\]
For each $U\in \bU(d)$, the tangent space $\bU(d)$ at the point $U$ is
\[
T_U\bU(d):=\{AU\,:\,A\in\fsu(d)\}.
\]

Throughout the present paper, we designate by $\|\cdot\|$ the operator norm on $M_d(\bbc)$ and by $\|\cdot\|_2$ the Frobenius norm on $M_d(\bbc)$, i.e.
\[
\begin{aligned}
\|A\|^2:=\text{ spectral radius of }A^*A,\,\,\hbox{ and }\,\,\|A\|_2^2:=\Tr(A^*A)
\end{aligned}
\]
for all $A\in M_d(\bbc)$. We shall use repeatedly the following obvious inequalities:
\begin{equation}\label{IneqNorms}
\|A\|\le\|A\|_2,\quad\|AB\|_2\le\|A\|\,\|B\|_2\,\,\text{ and }\,\,\|BA\|_2\le\|B\|_2\,\|A\|
\end{equation}
for all $A,B\in M_d(\bbc)$. We also recall that the Frobenius norm is associated with the Hermitian inner product $(A,B)\mapsto\Tr(A^*B)$ on $M_d(\bbc)$, and therefore satisfies the Cauchy-Schwarz inequality
\begin{equation}\label{CSFrob}
|\Tr(A^*B)|\le\|A\|_2\|B\|_2
\end{equation}
for all $A,B\in M_d(\bbc)$.

Let $X_1,X_2$ be two sets, and $\cM_1,\cM_2$ be $\si$-algebras of subsets of $X_1,X_2$ respectively. Let $\Phi:\,(X_1,\cM_1)\to(X_2,\cM_2)$ be a measurable map, and let $\nu_1$ be a positive (or signed, or complex)\footnote{We recall
that a complex (or a signed) measure must have finite total variation, and is therefore a complex linear combination of four finite positive measures.} measure on the measurable space $(X_1,\cM_1)$. The formula $\nu_2(B):=\nu_1(\Phi^{-1}(B))$ 
defines a positive (or signed, or complex) measure on the measurable space $(X_2,\cM_2)$, henceforth denoted as $\nu_2=\Phi\#\nu_1$, and referred to as the push-forward of $\nu_1$. 

If $X$ is a metric space, we designate by $\cP(X)$ the set of Borel probability measures on $X$. For each $p\ge 0$, we denote by $\cP_p(X)$ the set of Borel probability measures $\nu$ on $X$ such that
\[
\int_Xd(x_0,x)^p\nu(dx)<\infty
\]
for some $x_0\in X$, where $d$ is the metric on $X$. (By the triangle inequality, the set $\cP_p(X)$ is obviously independent of the choice of the point $x_0\in X$.)

Finally, we denote by $\fS_N$ the group of permutations of $\{1,\ldots,N\}$.


\subsection*{Main results and outline of the paper.}


Our first main result in this paper is the rigorous derivation of a mean-field kinetic equation which describes the behavior of the ``typical'' Lohe oscillator in a (large) system of $N\gg 1$ oscillators governed by the Lohe model \eqref{Lohe}.
Denote by $f\equiv f(t,dUdA)$ the probability that, at time $t$, the ``typical'' oscillator is in an infinitesimal box of size $dU$ about $U$, with $\mathrm{i}$ times ``frequency'' in an infinitesimal box of size $dA$ about $A$. We prove that, in
the large $N$ limit, the time-dependent probability measure $f$ is a weak solution to the following Vlasov-type, mean-field kinetic equation:
\begin{equation}\label{KLohe}
\left\{
\begin{aligned} 
{}&\d_tf+\Div_{U}\left(f(A+\tfrac{\ka}2(\la V\ra U^*-U\la V\ra^*))U\right)=0, 
\\
&\la V\ra:=\int_{\bU(d)\times\fsu(d)}Vf(t,dVdB), \quad (U, A)\!\in\!\bU(d)\!\times\!\fsu(d),\quad t > 0.
\end{aligned}
\right.
\end{equation}
(See Definition \ref{D3.1}, especially formula \eqref{C-0}, for definition of $\Div_U$.)

To the authors' knowledge, this result is the first rigorous derivation of a kinetic model on the group manifold $\bU(d) \times \fsu(d)$. We also estimate the fluctuations of the dynamics of the ``typical'' oscillator in \eqref{Lohe} about the
mean-field dynamics \eqref{KLohe}, and prove that they are of order $O(1/\sqrt{N})$ in quadratic Monge-Kantorovich (or Wasserstein) distance over any finite time interval. The rigorous derivation of the kinetic equation \eqref{KLohe} 
from the finite Lohe lattice \eqref{Lohe} occupies all of Section \ref{sec:3}, and is summarized in Theorem \ref{T3.1}.

Next we discuss the synchronization properties of the kinetic equation \eqref{KLohe}. If $\ka=0$, then solutions to \eqref{KLohe}Êsatisfy
\[
f(t)=\Phi_t\#f(0)\quad\hbox{ where }\Phi(t):\,(U,A)\mapsto(e^{tA}U,A).
\]
Unless $f$ is concentrated on some $A_0\in\fsu(d)$, this dynamics is dispersive, since, for different choices of $A$, the same $U$ is sent to different elements of $\bU(d)$ at the same time $t\not=0$ by the dynamics above. In other
words, the term $\Div_U(fAU)$ acts against synchronization. On the contrary, the coupling force in \eqref{KLohe}, i.e. $\la V\ra U^*-U\la V\ra^*$ vanishes if the matrix $S:=\la V\ra U^*$ is self-adjoint. For instance, if $S$ is positive 
definite, $U$ is the unitary matrix in the polar decomposition of $S$, and is therefore uniquely determined. This suggests that in the case $\ka\gg 1$, one should see that the dynamics of \eqref{KLohe} forces all the $U$ in the first
projection of $\Supp(f(t))$ in the cartesian product $\bU(d)\times\fsu(d)$ to concentrate in some sense around a single (possibly moving) element of $\bU(d)$.

In Section \ref{sec:4}, we state and prove a first result on synchronization for solutions to the Lohe mean-field, kinetic equation \eqref{KLohe} that are concentrated on $\bU(d)\times\{A_0\}$ for some $A_0\in\fsu(d)$, i.e. is of the form
\[
\rho(t)\otimes\de_{A_0}.
\]
Provided that the diameter of the support of $\rho(0))$ is smaller than some explicit threshold, we prove that $\Diam(\Supp(\rho(t)))\to 0$ as $t\to+\infty$, a property referred to as complete synchronization: see Theorem \ref{T-CSynchro}. 
We also provide a convergence rate for $\Diam(\Supp(\rho(t)))\to 0$ as $t\to+\infty$. In view of the heuristic argument above, the condition that $\Diam(\Supp(\rho)))$ is small enough implies that the average $\la V\ra$ is invertible, which
implies in turn that the polar decomposition of $\la V\ra$ is unique. For instance, one seeks to avoid the case where $\rho$ is the normalized Haar measure on $\bU(d)$, for which $\la V\ra=0$. 

In Section \ref{sec:5}, we state and prove a second result on synchronization for solutions to the Lohe mean-field, kinetic equation \eqref{KLohe}, assuming that the diameter of the second projection of $\Supp(f(t))$ in the cartesian 
product $\bU(d)\times\fsu(d)$ is smaller than some threshold proportional to the coupling strength $\kappa$. (The case studied in Section \ref{sec:4} corresponds to setting this diameter to $0$.) In the case where this diameter is 
positive but small enough, we establish a weaker variant of synchronization, called ``practical synchronization'' (see Definition \ref{D-PSynchro}). Specifically, assuming that $\Diam(\Supp(f^0))$ is small enough, we establish a bound 
of the form
\[
\varlimsup_{t\to+\infty}\iint_{(\bU(d)\times\fsu(d))^2}\|U-V\|_2^2f_\ka(t,dUdA)f_\ka(t,dVdB)=O(1/\ka)
\]
for the family $f_\ka$ of solutions to \eqref{KLohe} with initial data $f^0$, a compactly supported probability density on $\bU(d)\times\fsu(d)$. This indicates that, in the large time limit, the first projection of $\Supp(f_\ka(t))$ has a
tendency to shrink in the long time limit as the coupling strength $\ka$ increases to $+\infty$. In other words, even in the presence of the dispersing term $\Div_U(fAU)$, the interaction term in \eqref{KLohe} drives the mean-field
dynamics towards some form of synchronization.

Before discussing the mean field limit in Section \ref{sec:3}, we review some basic properties of the Lohe lattice in Section \ref{sec:2}, and explain how it is related to other models for synchronization, especially to the famous 
Kuramoto lattice \cite{Ku1,Ku2}. We have moved to the Appendix (a) the proof of existence and uniqueness of the solution to the Cauchy problem for the Lohe kinetic equation, and (b) one key estimate used to establish 
synchronization in both Section \ref{sec:4} and Section \ref{sec:5}.


\section{The Lohe Model and its Basic Properties} \label{sec:2}

\subsection{The Lohe matrix model} \label{sec:2.1} 


In \cite{Lo-1,Lo-2}, Lohe introduced a continuous dynamical system for $N$ ``noncommutative'' oscillators lying in $\bU(d)$. Denote by $U_j = U_j(t)\in M_d(\bbc)$ the state of the $j$-th Lohe oscillator. The dynamics of $U_j$ is governed by 
the following Cauchy problem:
\begin{equation}\label{Lohe-1}
\left\{\begin{aligned}  
\displaystyle {\mathrm i} \dot{U}_j U_j^*&=H_j + \frac{{\mathrm i} \kappa}{2N}\sum_{k= 1}^{N} \left(U_k U_j^* - U_j U_k^*  \right), \quad  t > 0, \quad j=1,\cdots,N, 
\\
U_j(0) &= U^0_j \in \bU(d),
\end{aligned}
\right.
\end{equation}
where $H_j=H^*_j\in M_d(\bbc)$, while $\kappa >0$ is the rescaled coupling strength (in units of the mean free time). For $U,V\in\bU(d)$, we also define a synchronization coupling between $U$ and $V$ via the formula
\[  
K(U,V) :=UV^*-VU^*.
\]
It is easy to see that
\[  
K(U, V)^*  = -K(U,V),\quad\text{ i.e. }\,\,K(U, V) \in \fsu(d). 
\]

We rewrite the Lohe model \eqref{Lohe-1} in the equivalent form:
\begin{equation} \label{Lohe-2}
{\dot U}_j= \left[ -{\mathrm i}H_j +\frac{\kappa}{2N}\sum_{k=1}^N K(U_k,U_j)  \right] U_j =\cA_j(U_1,\ldots,U_N)U_j,
\end{equation}
where $\cA_j$ is given by following formula:
\[
 \cA_j(U_1,\ldots,U_N):=-{\mathrm i}H_j +\frac{\kappa}{2N}\sum_{k=1}^N K(U_k,U_j).  
 \]
 Since $\cA_j$ is a quadratic function of its arguments for each $j=1,\ldots,N$, it is in particular a $C^1$ map on $M_d(\bbc)$, so that the Cauchy problem \eqref{Lohe-2} has a unique local solution for each $N$-tuple of initial data
 $U_1^0,\ldots,U_N^0\in M_d(\bbc)$. 
 
Obviously
\[  
\cA_j(U_1,\ldots,U_N)=-\cA_j(U_1,\ldots,U_N)^*\in\fsu(d),
\]
so that the group $\bU(d)$ is invariant under the Lohe flow defined by \eqref{Lohe-2}:
\[
U_j^0 \in\bU(d) \implies U_j(t)\in\bU(d)\quad\hbox{ for all }t\in\bbr \hbox{ and }~~j=1,\ldots,N.
\]
Since $\bU(d)$ is compact, this implies that the solution of the Cauchy problem for \eqref{Lohe-2} with initial data $U_1^0,\ldots,U_N^0\in\bU(d)$ is defined for all $t\in\bbr$. We refer to \cite{H-R} for the existence of phase-locked states 
and orbital stability of \eqref{Lohe}.

We next study the solution operator splitting for \eqref{Lohe-1} in the equal Hamiltonian case $H_j = H$ for $j=1,\ldots,N$:
\begin{equation} \label{Lohe-id}
\displaystyle \dot{U}_j  = -{\mathrm i} H U_j + \frac{\kappa}{2N}\sum_{k = 1}^{N} (U_k  - U_j U_k^* U_j ), \quad j =1, \cdots, N.
\end{equation}
Let ${\mathcal R}(t)$ and ${\mathcal L}(t)$ be the two solution operators corresponding to the following systems, respectively:
\begin{equation} \label{S-L}
\begin{cases}
\displaystyle \dot{V}_j  = -{\mathrm i} H V_j, 
\\
\displaystyle \dot{W}_j  =  \frac{\kappa}{2N}\sum_{k = 1}^{N} (W_k  - W_j W_k^* W_j ).
\end{cases}
\end{equation}
Denote
\[ 
{\mathcal R}(t) {\mathcal V}^0  := (e^{-{\mathrm i} H t} V_1^0, \cdots, e^{-{\mathrm i} H t} V_N^0), \qquad  {\mathcal L}(t) {\mathcal W}^0  := (W_1(t), \cdots, W_N(t)), \quad t \geq 0. 
\]
In that case, the solution operator to \eqref{Lohe-id} is the composition of the solution operators ${\mathcal S}(t)$ and ${\mathcal L}(t)$.

\begin{Prop} \label{P2.1}
Let ${\mathcal S}(t)$ be the solution operator to \eqref{Lohe-1}. Then
\[ 
{\mathcal S}(t) =  {\mathcal R}(t) \circ {\mathcal L}(t),\quad\hbox{Êfor all }t\in\bbr.
\]    
\end{Prop}

\begin{proof}
Let $(U_1,\ldots,U_N)$ be a solution to \eqref{Lohe-1}, and set 
\[
W_j(t) := e^{{\mathrm i} tH}U_j(t).
\]
Then
\[
\ba
\dot W_j(t)=&iHW_j(t)+e^{{\mathrm i} tH}\dot U_j(t)
\\
=&iHW_j(t)+e^{{\mathrm i} tH}\left[-iHU_j(t)+\frac{\kappa}{2N}\sum_{k = 1}^{N} (U_k(t)  - U_j(t) U_k(t)^* U_j(t) )\right]
\\
=&iHW_j(t)+\left[-iHW_j(t)+\frac{\kappa}{2N}\sum_{k = 1}^{N} (W_k(t)  - W_j(t) U_k(t)^* U_j(t) )\right]
\\
=&\frac{\kappa}{2N}\sum_{k = 1}^{N} (W_k(t)  - W_j(t) W_k(t)^* W_j(t) ),
\ea
\]
since
\[
W_k(t)^* W_j(t)=U_k(t)^*e^{-{\mathrm i} tH}e^{{\mathrm i} tH}U_j(t)=U_k(t)^*U_j(t).
\]
In other words, $(W_1,\ldots,W_N)$ satisfies $\eqref{S-L}_2$, i.e.
$$
(W_1(t),\ldots,W_N(t))=\mathcal{R}(-t)(U_1(t),\ldots,U_N(t))=\mathcal{L}(t)(U_1^0,\ldots,U_N^0),
$$
for all $t\in\bbr$, which is the sought identity.
\end{proof}

\subsection{Other synchronization models} \label{sec:2.2} 


In this subsection, we briefly discuss how the Lohe model is related to other synchronization models known in literature. 

We first consider the case of dimension $d=1$. In this case, for each $j=1,\ldots,N$, the unitary matrix $U_j$ is a complex number with modulus one, while $H_j$ is real number. Thus, we set
\begin{equation}\label{ans}
U_j(t):= e^{-{\mathrm i} \theta_j(t)}, \quad H_j := \nu_j,\qquad\theta_j(t),\,\,\nu_j \in \bbr,
\end{equation}
and substitute the ansatz \eqref{ans} into \eqref{Lohe} to obtain
\[
{\dot \theta}_j = \nu_j + \frac{\kappa}{N} \sum_{k=1}^{N} \sin (\theta_k - \theta_j), \quad j =1, \cdots, N,
\]
which is the well-known Kuramoto model  \cite{A-B, Ku1, Ku2}. In other words, the Lohe matrix model \eqref{Lohe} can be viewed as a nonabelian generalization of the Kuramoto model.

\medskip
As a next level of simplification, we consider the case $d=2$ in \eqref{Lohe}. In this case, we use the following parametrization of $U_i$ in terms of Pauli's matrices $\{ \sigma_k \}_{k=1}^{3}$:
\[ 
U_j:= e^{-{\mathrm i} \theta_j} \Big( {\mathrm i} \sum_{k=1}^{3} x_j^k \sigma_k + x_j^4 I_2  \Big) = e^{-{\mathrm i} \theta_j}
\left( \begin{array}{cc}
x_j^4 + {\mathrm i} x_j^1 & x_j^2 + {\mathrm i} x_j^3  
\\
-x_j^2 + {\mathrm i} x_j^3 & x_j^4 - {\mathrm i} x_j^1
\\
\end{array} \right), 
\]
where $I_2$ and $\sigma_n$ with $n=1,2,3$ are the identity matrix and Pauli matrices, respectively, defined by
\[
I_2 := \left( \begin{array}{cc}
  1 & 0 \\
  0 & 1 \\
  \end{array} \right), \quad 
  \sigma_1 := \left( \begin{array}{cc}
  1 & 0 \\
  0 & -1 \\
  \end{array} \right), \quad 
  \sigma_2 :=  \left( \begin{array}{cc}
  0 & -{\mathrm i} \\
  {\mathrm i} & 0 \\
  \end{array} \right), \quad 
  \sigma_3 := \left( \begin{array}{cc}
  0 & 1 \\
  1 & 0 \\
  \end{array} \right). 
\]
In this model, the matrices $U_j(t)$ are not necessarily unitary
We also expand $H_i$:
\[ 
H_j = \sum_{n=1}^{3} \omega_j^n \sigma_n + \nu_j I_2, 
\]
where $\omega_j = (\omega_j^1, \omega_j^2, \omega_j^3)\in\bbr^3$, and the natural frequency $\nu_j$ is associated with the $\bU(1)$ component of $U_j$, i.e. $e^{-\mathrm{i}\theta_j(t)}$. 

After some tedious algebraic manipulations, we obtain $5N$ equations for the angles $\theta_j$ and the four-vectors $x_j$:
\begin{align}
\begin{aligned} \label{Aug-Lohe}
||x_j||^2 {\dot \theta}_j&= \nu_j + \frac{\kappa}{N} \sum_{k=1}^{N} \sin (\theta_k - \theta_j) \langle x_j|x_k \rangle,~~1 \leq j \leq N,~t \in \bbr, 
\\
||x_j||^2 {\dot x}_j &= \Omega_j x_j + \frac{\kappa}{N} \sum_{k=1}^{N} \cos(\theta_k - \theta_j) (||x_j||^2 x_k -\langle x_j|x_k \rangle x_j),
\end{aligned}
\end{align}
where $\Omega_j$ is a real $4 \times 4$ skew-symmetric matrix:
\[ 
\Omega_i := \left( \begin{array}{cccc}
0 & -\omega_i^3 & \omega_i^2 & -\omega_i^1
\\
\omega_i^3 & 0 & -\omega_i^1 & -\omega_i^2 
\\
-\omega_i^2 & \omega_i^1 & 0 & -\omega_i^3 
\\
\omega_i^1 & \omega_i^2 &\omega_i^3 & 0 
\\
\end{array}
\right),
\]
and where $\langle\cdot|\cdot\rangle$ is the standard inner product in $\bbr^4$, while $\|x\|^2=\langle x|x\rangle$.

We take  $\theta_j = 0$ and $\nu_j = 0$ in \eqref{Aug-Lohe} to obtain the swarming model \cite{Ol}:
\begin{equation} \label{N-1}
||x_j||^2 {\dot x}_j = \Omega_j x_j + \frac{\kappa}{N} \sum_{k=1}^{N} ( ||x_j||^2 x_k - \langle x_j|x_k \rangle x_j), \quad j =1, \cdots, N.
\end{equation}

The emergence of phase-locked states for systems \eqref{Aug-Lohe} and \eqref{N-1} have already been investigated by the second author and his collaborators, e.g., \cite{C-C-H, C-H5}. 

For the generalizations of this type of model on some Riemannian manifolds and for a gradient flow formulation of the Lohe matrix model, we refer to \cite{H-K-R1, H-K-R2}.


\section{The Mean-Field Limit for the Lohe Matrix Model} \label{sec:3}


In this section, we present a derivation of a kinetic Lohe equation on the single-oscillator phase space $\bU(d)\times\fsu(d)$ from the Lohe matrix model. Specifically, the Lohe kinetic equation governs the mean-field limit of the Lohe matrix 
model in the large $N$ (number of oscillators) limit. We also provide a convergence rate estimate for this mean-field limit. 

\subsection{Preliminaries} \label{ssec:Prelim} 


For later use, we first recall some basic elements of analysis on the Lie group $\bU(d)$.

\begin{Def} \label{D3.1} A continuous vector field $X$ on $\bU(d)$ is a continuous section of the tangent bundle $T\bU(d)$, or, equivalently, a map $X:\,\bU(d)\to M_d(\bbc)$ of the form
\[
X_U=A(U)U\,,\quad\hbox{ where }~A: \bU(d)\ni U\mapsto A(U)\in\fsu(d)~~\hbox{is continuous}.
\]
For each $f\in\cP(\bU(d))$ and each continuous vector field $X$ on $\bU(d)$, one defines $\Div(fX)$ as the linear functional on $C^1(\bU(d))$ such that
\begin{equation} \label{C-0}
\la\Div(fX),\phi\ra :=-\int_{\bU(d)}(d_U\phi,X_U)f(dU)
\end{equation}
for each test function $\phi\in C^1(\bU(d))$. In this formula, $\la\cdot,\cdot\ra$ designates the duality between distribution densities and test functions on $\bU(d)$, while $(\cdot,\cdot)$ designates the duality between the space $T^*_U \bU(d)$ 
of tangent covectors and the space $T_U \bU(d)$ of tangent vectors to $\bU(d)$ at the point $U$.
\end{Def}

\smallskip
We shall often encounter the special case where the test function is of the form $\phi(U):=\|U-V\|_2^2$, where $V\in\bU(d)$ is given. Then
\begin{equation}\label{R3.1}
\la\Div(fX),\phi\ra=\int_{\bU(d)}\Tr(V^*A(U)U-U^*A(U)V)f(dU)\,.
\end{equation}
Indeed, applying \eqref{C-0} shows that
\[
\la\Div(fX),\phi\ra=-\int_{\bU(d)}\Tr(X^*_U(U-V)+(U-V)^*X_U)f(dU)\,.
\]
Since $A(U)^*=-A(U)$, one has
\begin{align*}
\begin{aligned}
& \Tr(X^*_U(U-V)+(U-V)^*X_U)
\\
& \hspace{0.5cm} =\Tr(U^*A(U)^*(U-V)+(U-V)^*A(U)U)
\\
& \hspace{0.5cm} =\Tr(U^*A(U)V-V^*A(U)U),
\end{aligned}
\end{align*}
which implies \eqref{R3.1}.

\subsection{The Lohe kinetic equation} \label{ssec:LoheKinEq} 


In this section, we give a ``formal'' derivation of the kinetic Lohe equation from the Lohe matrix model \eqref{Lohe} by a closure ansatz on the first equation in the BBGKY hierarchy. 

We first introduce the Liouville equation associated to the Lohe matrix model. Let
\[ 
F_N:\equiv F_N(t,dU_1dA_1\ldots dU_NdA_N)
\]
be a  time-dependent Borel probability measure on $\bU(d)^N\times\fsu(d)^N$ (where $t$ is the time variable). The Liouville equation is
\begin{equation} \label{Lio}
\d_tF_N+\sum_{j=1}^N\Div_{U_j}(F_NA_jU_j)+\frac{\kappa}{2N}\sum_{j=1}^N\Div_{U_j}\left(F_N\sum_{k=1}^N K(U_k,U_j)U_j\right)=0,
\end{equation}
where $A_j:=-{\mathrm i}H_j=-A_j^*$ for $j=1,\ldots,N$. The Liouville equation \eqref{Lio} is related to the Lohe matrix model by the following prescription. Call $\mathcal{S}[A_1,\ldots,A_N](t)$ the solution operator of \eqref{Lohe-1} with 
$A_j=-{\mathrm i}H_j$ for $j=1,\ldots,N$, and set
\[
\tilde{\mathcal{S}}(t):\,(U_1^0,A_1,\ldots,U_N^0,A_N)\mapsto(U_1(t),A_1,\ldots,U_N(t),A_N),
\]
where
\[
(U_1(t),\ldots,U_N(t)):=\mathcal{S}[A_1,\ldots,A_N](t)(U_1^0,\ldots,U_N^0).
\]

Then $t\mapsto F(t,dU_1dA_1,\ldots,dU_NdA_N)$ is a (weak) solution to the Liouville equation if and only if 
\[
F_N(t)=\tilde{\mathcal{S}}(t)\# F_N(0).
\]
(Apply Lemmas 8.1.4, 8.1.6 and Proposition 8.1.7 in \cite{AmbrosioGigliSavare}, after observing that the map
\[
U_j\mapsto A_jU_j+\tfrac{\kappa}{2N}\sum_{k=1}^NK(U_k,U_j)U_j
\]
is Lipschitz continuous on $\bU(d)$ for each $A_j\in M_d(\bbc)$.)

For each $\si \in \fS_N$, let $\cT_\si$ be the transformation defined on $(\bU(d)\times\fsu(d))^N$ by the formula:
\[
\cT_\si(U_1,,A_1,\ldots,U_N,A_N):=(U_{\si^{-1}(1)},A_{\si^{-1}(1)},\ldots,U_{\si^{-1}(N)},A_{\si^{-1}(N)}). 
\]

\begin{Lem} \label{L3.1}
For each $\si\in \fS_N$ and each $t\ge 0$, we have
\[
\cT_\si\#F_N(0)=F_N(0)\implies\cT_\si\#F_N(t)=F_N(t)\quad\hbox{ for all }t\ge 0.
\]
\end{Lem}

\begin{proof}
If $(U_1(t),\ldots,U_N(t))$ is a solution to \eqref{Lohe-2}, then
\[
\begin{aligned}
{\dot U}_{\si^{-1}(j)}=& A_{\si^{-1}(j)}U_{\si^{-1}(j)} +\frac{\kappa}{2N}\sum_{k=1}^N (U_k-U_{\si^{-1}(j)}U_k^*U_{\si^{-1}(j)})
\\
=& A_{\si^{-1}(j)}U_{\si^{-1}(j)} +\frac{\kappa}{2N}\sum_{k=1}^N (U_{\si^{-1}(k)}-U_{\si^{-1}(j)}U_{\si^{-1}(k)}^*U_{\si^{-1}(j)}),
\end{aligned}
\]
since
\[
\sum_{k=1}^N U_{\si^{-1}(k)}=\sum_{k=1}^N U_k\quad\hbox{ and }\quad\sum_{k=1}^N U_{\si^{-1}(k)}^*=\sum_{k=1}^N U_k^*.
\]
Therefore
\begin{align*}
\begin{aligned}
& \mathcal{S}[A_1,\ldots,A_N](t)(U^0_1,\ldots,U^0_N)=(U_1(t),\ldots,U_N(t))
\\
& \implies\mathcal{S}[A_{\si^{-1}(1)},\ldots,A_{\si^{-1}(N)}](t)(U^0_{\si^{-1}(1)},\ldots,U^0_{\si^{-1}(N)}) =(U_{\si^{-1}(1)}(t),\ldots,U_{\si^{-1}(N)}(t)).
\end{aligned}
\end{align*}
In other words,
\[
\begin{aligned}
\tilde{\mathcal{S}}(t)(\mathcal{T}_\si(U^0_1,A_1,\ldots,U^0_N,A_N))=&(U_{\si^{-1}(1)}(t),A_1,\ldots,U_{\si^{-1}(N)}(t),A_N)
\\
=&\mathcal{T}_\si(\mathcal{S}(t)(U^0_1,A_1,\ldots,U^0_N,A_N)),
\end{aligned}
\]
so that
\[
\cT_\si\#F_N(t)=\cT_\si\#(\mathcal{S}(t)\#F_N(0))=\mathcal{S}(t)\#(\cT_\si\#F_N(0))=\mathcal{S}(t)\#F_N(0)=F_N(t)
\]
for all $t\in\bbr$.
\end{proof}

\smallskip
Next, we consider the sequence of marginals of $F_N(t)$, defined as follows. 

\begin{Def}
Let $P_N\in\cP((\bU(d)\times\fsu(d))^N)$ satisfy $\mathcal{T}_\si\#P_N=P_N$ for each $\si\in\fS_N$. For each $m=1,\ldots,N-1$, the $m$-th marginal of $P_N$, denoted by $P_{N:m}$, is the element of $\cP(\bU(d)^m\times\fsu(d)^m)$ defined 
by the formula
\begin{equation}\label{C-4}
\begin{aligned}
& \int_{(\bU(d)\times\fsu(d))^m}\phi(U_1,A_1,\ldots,U_m,A_m)P_{N:m}(dU_1dA_1\ldots dU_mdA_m) 
\\
&\hspace{0.5cm} =\int_{(\bU(d)\times\fsu(d))^N}\phi(U_1,A_1,\ldots,U_m,A_m)P_{N}(dU_1dA_1\ldots dU_NdA_N).
\end{aligned}
\end{equation}
\end{Def}

Next, we derive an equation for $F_{N:1}$ (the first equation in the BBGKY hierarchy). Multiplying each side of the Liouville equation \eqref{Lio} by $\phi(U_1,A_1)$ and integrating both sides of the resulting identity in $U_1$ and $A_1$, we
arrive at
\begin{align*}
\begin{aligned} 
& \frac{d}{dt}\int_{\bU(d)\times\fsu(d)} \phi(U_1,A_1)F_{N:1}(t,dU_1dA_1)
\\
& \hspace{0.5cm} =\int_{\bU(d)\times\fsu(d)}(d_{U_1}\phi,A_1U_1)F_{N:1}(t,dU_1dA_1)
\\
& \hspace{0.7cm} +\!\frac{\kappa}{2N}\!\int_{\bU(d)^N\times\fsu(d)^N}\!\!\left(d_{U_1}\phi,\sum_{k=1}^N K(U_k,U_1)U_1\!\right)\!F_N(t,dU_1\ldots dA_N).
\end{aligned}
\end{align*}
Assuming that $\cT_\si\#F_N(0)=F_N(0)$ for each $\si\in\fS_N$, and applying Lemma \ref{L3.1}, we see that
\begin{align}
\begin{aligned}  \label{C-5}
& \int_{(\bU(d)\times\fsu(d))^N} \left(d_{U_1}\phi, K(U_k,U_1)U_1\right)F_N(t,dU_1dA_1\ldots dU_NdA_N) 
\\
& \hspace{0.5cm} =\int_{\bU(d)^N\times\fsu(d)^N}\left(d_{U_1}\phi,K(U_2,U_1)U_1\right)F_N(t,dU_1dA_1\ldots dU_NdA_N) 
\\
& \hspace{0.5cm} =\int_{(\bU(d)\times\fsu(d))^2}\left(d_{U_1}\phi,K(U_2,U_1)U_1\right)F_{N:2}(t,dU_1dA_1dU_2dA_2).
\end{aligned}
\end{align}
Hence, it follows from \eqref{C-4} and \eqref{C-5} that
\[
\begin{aligned} 
\frac{d}{dt}\int&_{\bU(d)\times\fsu(d)}\phi(U_1,A_1)F_{N:1}(t,dU_1dA_1)
\\
=&\int_{\bU(d)\times\fsu(d)}(d_{U_1}\phi,A_1U_1)F_{N:1}(t,dU_1dA_1)
\\
&+\frac{\kappa (N-1)}{2N}\int_{(\bU(d)\times\fsu(d))^2}\left(d_{U_1}\phi, K(U_2,U_1)U_1\right)F_{N:2}(t,dU_1dA_1dU_2dA_2).
\end{aligned}
\]
Equivalently, we have\footnote{So far, we have used the notation $P_{2:1}$ to designate the first marginal of a symmetric Borel probability measure on $(\bU(d)\times\fsu(d))^2$. The term $F_{N:2} K(U_2,U_1)U_1$ is a matrix 
whose entries are complex measures on $(\bU(d)\times\fsu(d))^2$, and we designate by $(F_{N:2} K(U_2,U_1)U_1)_{:1}$ the matrix whose entries are the push-foward of the entries of $F_{N:2} K(U_2,U_1)U_1$ by the projection
map 
$(U_,A_1,U_2,A_2)\mapsto(U_1,A_1)$.}
\[
\d_tF_{N:1}+\Div_{U_1}(F_{N:1}A_1U_1)+\frac{\kappa (N-1)}{2N}\Div_{U_1}\left(F_{N:2} K(U_2,U_1)U_1\right)_{:1}=0.
\]
This is the first equation in the BBGKY hierarchy. Observe that this is not a closed equation for $F_{N:1}$ since it involves $F_{N:2}$.  

We now assume that
\[
\begin{aligned}
F_{N:1}(t,dU_1dA_1)&\wto f(t,dU_1dA_1),
\\
F_{N:2}(t,dU_1dU_2dA_1dA_2)&\wto f(t,dU_1dA_1)f(t,dU_2dA_2)
\end{aligned}
\]
as $N\to\infty$ in $L^\infty(\bbr;\cP(\bU(d)\times\fsu(d)))$ and $L^\infty(\bbr;\cP((\bU(d)\times\fsu(d))^2))$ weak-$*$ topology respectively. Then, the single-particle distribution function $f\equiv f(t,dUdA)$ satisfies
\begin{equation} \label{V-M}
\d_tf+\Div_{U}(fAU)+\frac{\kappa}{2}\Div_U\left(f\int_{\bU(d)}K(V,U)U\rho_f(t,dV)\right)=0,
\end{equation}
in ${\mathcal D}'(\bbr\times\bU(d)\times\fsu(d))$, where $\rho_f(t)$ is the Borel probability measure on $\bU(d)$ defined by 
\begin{equation}\label{DefRhof}
\int_{\bU(d)}\phi(U)\rho_f(t,dU)=\int_{\bU(d)\times\fsu(d)}\phi(U)f(t,dUdA)
\end{equation}
for all $\phi\in C(\bU(d))$. In other words, $\rho_f(t)$ is the push-forward of $f(t)$ by the map $\bU(d)\times\fsu(d)\ni(U,A)\mapsto U\in\bU(d)$.

The existence and uniqueness theory of a solution to the Cauchy problem for the Lohe kinetic equation \eqref{V-M} is summarized in the following proposition.

\begin{Prop}\label{P-LoheKinEq}
For all $f^0\in\cP_1(\bU(d)\times\fsu(d))$, the Cauchy problem for the kinetic Lohe equation \eqref{V-M} has a unique weak solution $f\in C([0,+\infty),(\cP_1(E),\MK))$ with initial data $f^0$.

Moreover, if $f^0\in\cP_p(\bU(d)\times\fsu(d))$ for some $p\ge 1$, then $f(t)\in\cP_p(\bU(d)\times\fsu(d))$ for all $t\ge 0$.
\end{Prop}

The proof of this proposition will be given in Appendix \ref{App:A}.

\smallskip
An important special case of the setting considered above is the case where
\[
\int_{\Om}f(0,dUdA)=0
\]
for each negligible set $\Omega$ in $\bU(d)\times\fsu(d)$ (viewed as a finite dimensional smooth manifold). In that case, set $\mu$ to be the normalized Haar measure\footnote{Except in Section \ref{sec:4}, one could equivalently replace 
$\mu$ with any Borel probability measure on $\bU(d)$ such that $\mu(O)>0$ for each open set $O\subset\bU(d)$ such that $O\not=\varnothing$.}, and let $\g$ be an arbitrary Borel probability measure on the linear space $\fsu(d)$ with 
positive density with respect to the Lebesgue measure. For instance, one can assume that $\g$ is the Gaussian measure on $\fsu(d)$ such that the real and imaginary parts of each subdiagonal entry and the imaginary part of each diagonal 
entry of an element of $\fsu(d)$ distributed under $\g$ are independent and distributed under the centered, reduced Gaussian probability measure on the real line.

Then $f(0)$ is absolutely continuous with respect to $\mu\otimes\g$, and therefore is of the form $f^{0}(U,A)\mu(dU)\g(dA)$, where $f^0$ is a probability density on $\bU(d)\times\fsu(d)$, by the Radon-Nikodym theorem. In that case, for all 
$t\ge 0$, one has
\[
f(t,dUdA)=f(t,U,A)\mu(dU)\g(dA), \quad t \geq 0, 
\]
where the function $f$ satisfies the equation
\begin{equation} \label{K-Lohe}
\d_tf+\Div_{U}(fAU)+\frac{\kappa}{2}\Div_U\left(f\int_{\bU(d)} K(V,U)U\rho_f(t,V)\mu(dV)\right)=0,
\end{equation}
with
\[
\rho_f(t,U):=\int_{\fsu(d)}f(t,U,A)\g(dA).
\]

\subsection{Convergence rate} \label{sec:3.2}


In this subsection, we study the convergence from $F_{N:1}$ to $f$ in the Monge-Kantorovich (or Wasserstein) distance of exponent $2$.  For the reader's convenience, we briefly recall the basic notions pertaining to this distance. 

\begin{Def}
Let $E$ be a (finite dimensional) unitary space, with norm denoted $\|\cdot\|$, and let $m_1,m_2\in\cP_2(E)$. Denote by $\Pi(m_1,m_2)$ the set of couplings of $m_1$ and $m_2$, i.e. the set of Borel probability measures $m$ on $E\times E$ 
such that
\[
\iint_{E\times E}(\phi(x)+\psi(y))m(dxdy)=\int_E\phi(x)m_1(dx)+\int_E\psi(y)m_2(dy)
\]
for all $\phi,\psi\in C_b(E)$. Then, the expression\footnote{Notice that $m_1\otimes m_2\in\Pi(m_1,m_2)$, so that $\Pi(m_1,m_2)\not=\varnothing$.} 
\[
\MKd(m_1,m_2):=\left(\inf_{m\in\Pi(m_1,m_2)}\iint_{E\times E}\|x-y\|^2m(dxdy)\right)^{1/2}\in[0,+\infty)
\]
defines a distance on the set of Borel probability measures on $E$ with finite second order moment.
\end{Def}

For more information on the distance $\MKd$, see Chapter 7 of the book \cite{VillaniAMS}, and Chapter 5 and Chapter 7 in \cite{AmbrosioGigliSavare}.

\medskip
Notice first that $f(t)^{\otimes N}$ satisfies
\[
\d_tf^{\otimes N}+\sum_{j=1}^N\Div_{U_j}(f^{\otimes N}A_jU_j)+\frac{\kappa}{2} \sum_{j=1}^N\Div_{U_j}\left(f^{\otimes N}\int_{\bU(d)} K(\hat U,U_j)U_j\rho_f(t,d\hat U)\right)=0,
\]
where $\rho_f$ is the average of $f$ in the variable $A$ defined in \eqref{DefRhof}. In fact, the following stronger property holds true.

\begin{Lem}\label{L-UniqVlasN}
Let $t\mapsto f(t)$ be a continuous map with values in the set of Borel probability measures on $\bU(d)\times\fsu(d)$ equipped with the weak topology of probability measures. The Cauchy problem
\[
\left\{
\begin{aligned}
{}&\d_tg_N+\sum_{j=1}^N\Div_{U_j}(g_NA_jU_j)+\frac{\kappa}{2}\sum_{j=1}^N\Div_{U_j}\left(g_N\int_{\bU(d)}K(\hat U,U_j)U_j\rho_f(t,d\hat U)\right)=0,
\\
&g_N=g_N^0,
\end{aligned}
\right.
\]
has a unique weak solution $t\mapsto g_N(t)$ that is continuous with values in the set of Borel probability measures on $(\bU(d)\times\fsu(d))^N$ equipped with the weak topology of probability measures.
\end{Lem}

\begin{proof}
For each $(A_1,\ldots,A_N)$, the map 
\[
(t,U_j)\mapsto\left(A_j+\tfrac{\kappa}2\int_{\bU(d)}K(\hat U,U_j)\rho_f(t,d\hat U)\right)U_j
\]
is continuous on $[0,+\infty)\times\bU(d)$, and Lipschitz continuous in $U_j$ uniformly in $t\ge 0$. The existence and uniqueness of $g_N$ follows from Lemmas 8.1.4, 8.1.6 and Proposition 8.1.7 in \cite{AmbrosioGigliSavare}.
\end{proof}

Let
\[
Q_N^{0}\in\Pi((f^0)^{\otimes N},(f^0)^{\otimes N}),
\]
and let $t\mapsto Q_N(t)\in\cP((\bU(d)\times\fsu(d))^N\times(\bU(d)\times\fsu(d))^N)$ be the weak solution of the Cauchy problem:
\begin{equation}\label{New-1}
\left\{\begin{aligned}
{}&\d_tQ_N+\displaystyle\sum_{j=1}^N\Div_{U_j}(Q_NA_jU_j)+\sum_{j=1}^N\Div_{V_j}(Q_NB_jV_j)
\\
&\hskip .85cm\displaystyle+\frac{\kappa}{2}\sum_{j=1}^N\Div_{U_j}\left(Q_N\int_{\bU(d)} K(\hat U,U_j)U_j\rho_f(t,d\hat U)\right) 
\\
&\hskip .85cm\displaystyle+\frac{\kappa}{2}\sum_{j=1}^N\Div_{V_j}\left(Q_N\frac1N\sum_{k=1}^N K(V_k,V_j)V_j\right)=0, \qquad\quad t > 0, 
\\
&Q_N\rstr_{t=0}=Q_N^{0}.
\end{aligned}
\right.
\end{equation}

\begin{Lem}\label{PropQN}
For each $t\ge 0$,
\[
Q_N(t)\in\Pi(f(t)^{\otimes N},F_N(t)).
\]
Moreover, for each $\si\in\fS_N$, set
\begin{align*}
\begin{aligned}
& R_\si(U_1,A_1,\ldots,U_N,A_N,V_1,B_1,\ldots,V_N,B_N) \\
&:= (U_{\si^{-1}(1)},A_{\si^{-1}(1)},\ldots,U_{\si^{-1}(N)},A_{\si^{-1}(N)},V_{\si^{-1}(1)},B_{\si^{-1}(1)},\ldots,V_{\si^{-1}(N)},B_{\si^{-1}(N)}).
\end{aligned}
\end{align*}
Then
\[
R_\si\#Q_N^0=Q_N^0\implies R_\si\#Q_N(t)=Q_N(t)\quad\hbox{ for all }t\ge 0\text{ and all }\si\in\fS_N.
\]
\end{Lem}

\begin{proof}
Call $P_1$ and $P_2$ the projections
\[
\begin{aligned}
P_1:\,(U_1,A_1,\ldots,U_N,A_N,V_1,B_1,\ldots,V_N,B_N)\mapsto(U_1,A_1,\ldots,U_N,A_N),
\\
P_2:\,(U_1,A_1,\ldots,U_N,A_N,V_1,B_1,\ldots,V_N,B_N)\mapsto(V_1,B_1,\ldots,V_N,B_N).
\end{aligned}
\]

Let $\phi\in C^1_c((\bU(d)\times\fsu(d))^N)$. Then
\begin{align*}
\begin{aligned}
& \frac{d}{dt}\int\phi(U_1,\ldots,A_N)Q_N(t,dU_1\ldots dB_N)
\\
& \hspace{0.5cm} =\int\sum_{j=1}^N(d_{U_j}\phi(U_1,\ldots,A_N),A_jU_j)Q_N(t,dU_1\ldots dB_N)
\\
& \hspace{0.7cm}+\frac{\kappa}{2}\int\sum_{j=1}^N\left(d_{U_j}\phi(U_1,\ldots,A_N),\int_{\bU(d)}K(\hat U,U_j)\rho_f(t,d\hat U)U_j\right)Q_N(t,dU_1\ldots dB_N),
\end{aligned}
\end{align*}
which is the weak formulation of 
\[
\begin{aligned}
\d_tP_1\#Q_N+\sum_{j=1}^N\Div_{U_j}\left(P_1\#Q_N\left(A_j+\frac{\kappa}{2}\int_{\bU(d)}K(\hat U,U_j)\rho_f(t,d\hat U)\right)U_j\right)=0.
\end{aligned}
\]
Since $P_1\#Q_N(0)=f(0)^{\otimes N}$, we conclude from the uniqueness part in Lemma \ref{L-UniqVlasN} that
\[
P_1\#Q_N(t)=f(t)^{\otimes N}\quad\hbox{ for all }t\ge 0.
\]

Likewise
\begin{align*}
\begin{aligned}
& \frac{d}{dt}\int\phi(V_1,\ldots,B_N)Q_N(t,dU_1\ldots dB_N)
\\
& \hspace{0.5cm} =\int\sum_{j=1}^N(d_{V_j}\phi(U_1,\ldots,A_N),B_jV_j)Q_N(t,dU_1\ldots dB_N)
\\
& \hspace{0.7cm} +\frac{\kappa}{2}\int\sum_{j=1}^N\left(d_{V_j}\phi(U_1,\ldots,A_N),\frac1N\sum_{k=1}^NK(V_k,V_j)V_j\right)Q_N(t,dU_1\ldots dB_N),
\end{aligned}
\end{align*}
which is the weak formulation of
\[
\begin{aligned}
\d_tP_2\#Q_N+\sum_{j=1}^N\Div_{V_j}\left(P_2\#Q_N\left(B_j+\frac{\kappa}{2N}\sum_{k=1}^NK(V_k,V_j)\right)V_j\right)=0.
\end{aligned}
\]
Since $P_2\#Q_N(0)=f(0)^{\otimes N}$, we conclude from the uniqueness of the solution to the Cauchy problem for the Liouville equation \eqref{Lio} (which follows from Proposition 8.1.7 in \cite{AmbrosioGigliSavare}) that
\[
P_2\#Q_N(t)=F_N(t)\quad\hbox{ for all }t\ge 0.
\]
Thus $Q_N(t)\in\Pi(f(t)^{\otimes N},F_N(t))$ for each $t\ge 0$.

Next, observe that, for each $\psi\in C^1_c((\bU(d)\times\fsu(d))^N\times(\bU(d)\times\fsu(d))^N)$ and each $\si\in\fS_N$, one has
\begin{align*}
\begin{aligned}
& \frac{d}{dt}\int\psi\circ R_\si(U_1,\ldots,B_N)Q_N(t,dU_1\ldots dB_N)
\\
& \hspace{0.5cm} =\int\sum_{j=1}^N(d_{U_j}\psi\circ R_\si(U_1,\ldots,B_N),A_jU_j)Q_N(t,dU_1\ldots dB_N)
\\
& \hspace{0.7cm} +\int\sum_{j=1}^N(d_{V_j}\psi\circ R_\si(U_1,\ldots,B_N),B_jV_j)Q_N(t,dU_1\ldots dB_N)
\\
& \hspace{0.5cm}+\frac{\kappa}{2}\int\sum_{j=1}^N\left(d_{U_j}\psi\circ R_\si(U_1,\ldots,B_N),K(\hat U,U_j)U_j\right)\rho_f(t,d\hat U)Q_N(t,dU_1\ldots dB_N)
\\
& \hspace{0.5cm}+\frac{\kappa}{2}\int\sum_{j=1}^N\left(d_{V_j}\psi\circ R_\si(U_1,\ldots,B_N),\frac1N\sum_{k=1}^NK(V_k,V_j)V_j\right)Q_N(t,dU_1\ldots dB_N),
\end{aligned}
\end{align*}
so that
\[
\begin{aligned}
\d_t(R_\si\#Q_N)&+\sum_{j=1}^N\Div_{U_{\si(j)}}((R_\si\#Q_N)A_{\si(j)}U_{\si(j)})
\\
&+\sum_{j=1}^N\Div_{V_{\si(j)}}((R_\si\#Q_N)B_{\si(j)}V_{\si(j)})
\\
&+\frac{\kappa}{2}\sum_{j=1}^N\Div_{U_{\si(j)}}\left((R_\si\#Q_N)\int_{\bU(d)} K(\hat U,U_{\si(j)})U_{\si(j)}\rho_f(t,d\hat U)\right) 
\\
&+\frac{\kappa}{2}\sum_{j=1}^N\Div_{V_{\si(j)}}\left((R_\si\#Q_N)\frac1N\sum_{k=1}^N K(V_k,V_{\si(j)})V_{\si(j)}\right)=0.
\end{aligned}
\]
With the substitution $l=\si(j)$ in the four summations above, one concludes that $R_\si\#Q_N$ and $Q_N$ are both solutions of the same Cauchy problem \eqref{New-1} provided that $R_\si\#Q_N^0=Q_N^0$. Since the map
$K:\,\bU(d)\times\bU(d)\to M_d(\bbc)$ is Lipschitz continuous, the Cauchy problem \eqref{New-1} has a unique solution (according to Proposition 8.1.7 in \cite{AmbrosioGigliSavare}), and therefore $R_\si\#Q_N(t)=Q_N(t)$ for 
all $t\ge 0$.
\end{proof}

\medskip
Henceforth, we assume that
\[
R_\si\#Q_N^0=Q_N^0\quad\hbox{ for all }\si\in\fS_N.
\]
Following the strategy outlined in Section 3 of \cite{F-M-P}Ê and in Section 5 of \cite{F-P}, we seek to control the quantity
\[
J_N(t):=\frac1N\sum_{j=1}^N \int_{(\bU(d)\times\fsu(d))^{2N}} \left(\|U_j\!-\!V_j\|_2^2\!+\!\|A_j\!-\!B_j\|_2^2 \right)Q_N(t,dU_1\ldots,dB_N).
\]

\begin{Lem}\label{L-JNMKd}
For each $t\ge 0$, one has
\[
J_N(t)=\int_{(\bU(d)\times\fsu(d))^{2N}} \left(\|U_j\!-\!V_j\|_2^2\!+\!\|A_j\!-\!B_j\|_2^2 \right)Q_N(t,dU_1\ldots,dB_N)
\]
for all $j=1,\ldots,N$. In particular
\[
\begin{aligned}
J_N(t)&=\int_{(\bU(d)\times\fsu(d))^{2N}} \left(\|U_1\!-\!V_1\|_2^2\!+\!\|A_1\!-\!B_1\|_2^2 \right)Q_N(t,dU_1\ldots,dB_N)
\\
&\ge\MKd(f(t),F_{N:1}(t))^2\,.
\end{aligned}
\]
\end{Lem}

\begin{proof}
Let $\tau\in\fS_N$ be the transposition exchanging $j$ and $k$; by Lemma \ref{PropQN} 
\begin{align*}
\begin{aligned}
& \int_{(\bU(d)\times\fsu(d))^{2N}} \left(\|U_k\!-\!V_k\|_2^2\!+\!\|A_k\!-\!B_k\|_2^2 \right)Q_N(t,dU_1\ldots dB_N)
\\
& \hspace{0.5cm} =\int_{(\bU(d)\times\fsu(d))^{2N}} \left(\|U_j\!-\!V_j\|_2^2\!+\!\|A_j\!-\!B_j\|_2^2 \right)R_\tau\#Q_N(t,dU_1\ldots dB_N)
\\
& \hspace{0.5cm} =\int_{(\bU(d)\times\fsu(d))^{2N}} \left(\|U_j\!-\!V_j\|_2^2\!+\!\|A_j\!-\!B_j\|_2^2 \right)Q_N(t,dU_1\ldots dB_N).
\end{aligned}
\end{align*}
This proves the first identity in the lemma. 

As for the second identity, set
\[
p_1:\,(U_1,A_1,\ldots,U_N,A_N,V_1,B_1,\ldots,V_N,B_N)\mapsto(U_1,A_1,V_1,B_1),
\]
and observe that 
\[
p_1\#Q_N(t)\in\Pi(p_1\#f(t)^{\otimes N},p_1\#F_N(t))=\Pi(f(t),F_{N:1}(t)).
\]
Since
\[
\begin{aligned}
J_N(t)&=\int_{(\bU(d)\times\fsu(d))^{2N}}\left(\|U_1\!-\!V_1\|_2^2\!+\!\|A_1\!-\!B_1\|_2^2 \right)Q_N(t,dU_1\ldots dB_N)
\\
&=\int_{(\bU(d)\times\fsu(d))^2}\left(\|U_1\!-\!V_1\|_2^2\!+\!\|A_1\!-\!B_1\|_2^2 \right)p_1\#Q_N(t,dU_1dA_1dV_1B_1),
\end{aligned}
\]
the inequality $J_N(t)\ge\MKd(f(t),F_{N:1}(t))^2$ follows from the very definition of the distance $\MKd$.
\end{proof}

Straightforward computations\footnote{In these computations, and in similar computations appearing later in this paper, we systematically omit the domain of integration and the variables in $Q_N$ when there is no risk of ambiguity in order 
to avoid cluttered mathematical expressions.} show that
\begin{equation} \label{C-7}
\begin{aligned}
\frac{dJ_N}{dt} &= \int \left((d_{U_1}\|U_1-V_1\|_2^2,A_1U_1)+(d_{V_1}\|U_1-V_1\|_2^2,B_1V_1)\right)Q_N 
\\
&+ \int \left(d_{U_1}\|U_1-V_1\|_2^2,\frac{\kappa}{2}\int_{\bU(D)} K(\hat U,U_1)U_1\rho_f(t,d\hat U)\right)Q_N 
\\
&+ \int \left(d_{V_1}\|U_1-V_1\|_2^2,\frac{\kappa}{2N}\sum_{k=1}^N K(V_k,V_1)V_1\right)Q_N
\\
&= \int \left((d_{U_1}\|U_1-V_1\|_2^2,A_1U_1)+(d_{V_1}\|U_1-V_1\|_2^2,B_1V_1)\right)Q_N 
\\
&+ \frac{\kappa}{2}\int \left(d_{U_1}\|U_1-V_1\|_2^2,\int K(\hat U,U_1)U_1\rho_f(t,d\hat U)-\frac1N\sum_{k=1}^N K(U_k,U_1)U_1\right)Q_N 
\\
&+ \frac{\kappa}{2N}\int \left(d_{U_1}\|U_1-V_1\|_2^2,\sum_{k=1}^N K(U_k,U_1)U_1\right)Q_N 
\\
&+ \frac{\kappa}{2N}\int  \left(d_{V_1}\|U_1-V_1\|_2^2,\sum_{k=1}^N K(V_k,V_1)V_1\right)Q_N 
\\
&=: {\mathcal I}_{11}+ {\mathcal I}_{12} + {\mathcal I}_{13} + {\mathcal I}_{14}.
\end{aligned}
\end{equation}

In next lemma, we estimate the terms ${\mathcal I}_{1i}$.  \newline

\begin{Lem} \label{L3.3} 
Let $f$ be a classical solution to \eqref{V-M}. Then the terms  ${\mathcal I}_{1i}$ satisfy
\[ {\mathcal I}_{11} \le J_N, \qquad {\mathcal I}_{13} + {\mathcal I}_{14} \leq 8J_N, \qquad  {\mathcal I}_{12} \leq J_N +  \frac{16d}{N}.  \]
\end{Lem}

\begin{proof}  First, we recall Remark \ref{R3.1} (2), to prove that, if $X,Y\in\fsu(d)$, one has
\begin{align}\label{DiffNorm2}
\begin{aligned}
& (d_U\|U-V\|_2^2,XU)_U+(d_V\|U-V\|_2^2,YV)_V \\
& \hspace{0.5cm} =\Tr(U^*XV-V^*XU)+\Tr(V^*YU-U^*YV)
\\
& \hspace{0.5cm}=\Tr(U^*(X-Y)V-V^*(X-Y)U)
\\
& \hspace{0.5cm} =\Tr(U^*(X-Y)(V-U)+(U-V)^*(X-Y)U).
\end{aligned}
\end{align}

With this identity, we shall estimate separately ${\mathcal I}_{11}$ and ${\mathcal I}_{13} + {\mathcal I}_{14}$. \newline

\noindent 
$\bullet$ Case A (${\mathcal I}_{11}$). Applying \eqref{DiffNorm2} with $X=A_1$ and $Y=B_1$ shows that
\begin{align*}
\begin{aligned}
& (d_{U_1}\|U_1-V_1\|_2^2,A_1U_1)+(d_{V_1}\|U_1-V_1\|_2^2,B_1V_1)
\\
& \hspace{0.5cm} =
\Tr(U_1^*(A_1-B_1)(V_1-U_1)+(U_1-V_1)^*(A_1-B_1)U_1),
\end{aligned}
\end{align*}
so that, by the Cauchy-Schwarz inequality \eqref{CSFrob} and the second inequality in \eqref{IneqNorms}
\begin{align*}
\begin{aligned}
& |(d_{U_1}\|U_1-V_1\|_2^2,A_1U_1)_U|+|(d_{V_1}\|U_1-V_1\|_2^2,B_1V_1)_V|
\\
& \hspace{0.5cm} \le
(\|U_1^*(A_1-B_1)\|_2+\|(A_1-B_1)U_1\|_2)\|V_1-U_1\|_2
\\
& \hspace{0.5cm} \le 2\|A_1-B_1\|_2\|V_1-U_1\|_2\le\|A_1-B_1\|^2_2+\|V_1-U_1\|^2_2.
\end{aligned}
\end{align*}
Hence, by Lemma \ref{L-JNMKd}
\[  
{\mathcal I}_{11}\le J_N. 
\]

\smallskip
\noindent 
$\bullet$ Case B (${\mathcal I}_{13} + {\mathcal I}_{14}$).
Since $R_\si\#Q_N(t)=Q_N(t)$ by Lemma \ref{PropQN}
\[
\begin{aligned}
{\mathcal I}_{13}=\frac{N-1}{N}\int\left(d_{U_1}\|U_1-V_1\|^2,K(U_2,U_1)U_1\right)Q_N, 
\\
{\mathcal I}_{14}=\frac{N-1}{N}\int\left(d_{V_1}\|U_1-V_1\|^2,K(V_2,V_1)V_1\right)Q_N.
\end{aligned}
\]
Applying \eqref{DiffNorm2} with $X=K(U_2,U_1)$ and $Y=K(V_2,V_1)$ shows that
\begin{align*}
\begin{aligned}
& \left(d_{U_1}\|U_1-V_1\|_2^2,K(U_2,U_1)U_1\right)+\left(d_{V_1}\|U_1-V_1\|_2^2,K(V_2,V_1)V_1\right)
\\
& \hspace{0.5cm} =\Tr(U_1^*(K(U_2,U_1)-K(V_2,V_1))(V_1-U_1))
\\
& \hspace{0.5cm}  +\Tr((U_1-V_1)(K(U_2,U_1)-K(V_2,V_1))U_1),
\end{aligned}
\end{align*}
so that
\begin{align}\label{C-9}
\begin{aligned}
& |\left(d_{U_1}\|U_1-V_1\|_2^2,K(U_2,U_1)U_1\right)_U|+|\left(d_{V_1}\|U_1-V_1\|_2^2,K(V_2,V_1)V_1\right)_V|
\\
& \hspace{2cm} \le
2\|V_1-U_1\|_2\|(K(V_2,V_1)-K(U_1,U_2))\|_2,
\end{aligned}
\end{align}
again by the Cauchy-Schwarz inequality \eqref{CSFrob} and the second inequality in \eqref{IneqNorms}.

On the other hand
\begin{align*}
\begin{aligned}
& K(V_2,V_1)-K(U_1,U_2)=(V_2V_1^*-V_1V_2^*)-(U_1U_2^*-U_2U_1^*) \\
& \hspace{1.5cm}  =(V_2-U_2)V_1^*+U_2(V_1-U_1)^* -(V_1-U_1)V_2^*-U_1(V_2-U_2)^*
\end{aligned}
\end{align*}
so that
\begin{equation} \label{C-10}
\|(K(V_2,V_1)-K(U_1,U_2))\|_2\le 2\|V_2-U_2\|_2+2\|V_1-U_1\|_2.
\end{equation}
Combining \eqref{C-9} with \eqref{C-10} and using the elementary inequality $2ab\le a^2+b^2$ shows that
\[
\begin{aligned}
{\mathcal I}_{13} + {\mathcal I}_{14}\le& 2\int\|V_1-U_1\|_2(2\|V_2-U_2\|_2+2\|V_1-U_1\|_2)Q_N 
\\
\le&4\int\|V_1-U_1\|_2^2Q_N+4\int\|V_2-U_2\|_2\|V_1-U_1\|_2Q_N 
\\
\le&6\int\|V_1-U_1\|_2^2 Q_N+2\int\|V_2-U_2\|_2^2 Q_N\le8J_N.
\end{aligned}
\]

It remains to treat the term ${\mathcal I}_{12}$, which is handled similarly, with some minor differences.

\smallskip
\noindent $\bullet$ Case C (${\mathcal I}_{12}$). The integrand in ${\mathcal I}_{12}$ is put in the form
\begin{align*}
\begin{aligned}
& \left(d_{U_1}\|U_1-V_1\|_2^2,\left(\int_{\bU(d)}K(\hat U,U_1)\rho_f(t,d\hat U)-\frac1N\sum_{k=1}^NK(U_k,U_1)\right)U_1\right)
\\
& \hspace{0.5cm} =\Tr\left((U_1-V_1)^*\left(\int_{\bU(d)}K(\hat U,U_1)\rho_f(t,d\hat U)-\frac1N\sum_{k=1}^NK(U_k,U_1)\right)U_1\right)
\\
& \hspace{0.7cm} +\Tr\left(U_1^*\left(\int_{\bU(d)}K(\hat U,U_1)\rho_f(t,d\hat U)-\frac1N\sum_{k=1}^NK(U_k,U_1)\right)^*(U_1-V_1)\right)
\end{aligned}
\end{align*}
by \eqref{DiffNorm2} with $Y=0$ and 
\[
X=\int_{\bU(d)}K(\hat U,U_1)\rho_f(t,d\hat U)-\frac1N\sum_{k=1}^NK(U_k,U_1)\,.
\]
Hence
\[
\begin{aligned}
\left|\left(d_{U_1}\|U_1-V_1\|_2^2,\left(\int_{\bU(d)}K(\hat U,U_1)\rho_f(t,d\hat U)-\frac1N\sum_{k=1}^NK(U_k,U_1)\right)U_1\right)\right|
\\
\le\|U_1-V_1\|_2^2+\left\|\int_{\bU(d)}K(\hat U,U_1)\rho_f(t,d\hat U)-\frac1N\sum_{k=1}^NK(U_k,U_1)\right\|_2^2
\end{aligned}
\]
by the Cauchy-Schwarz inequality \eqref{CSFrob} and the basic inequality $2ab\le a^2+b^2$. Hence
\[
\begin{aligned} 
{\mathcal I}_{12}&\le J_N+\int\left\|\int_{\bU(d)}K(\hat U,U_1)\rho_f(t,d\hat U)-\frac1N\sum_{k=1}^NK(U_k,U_1)\right\|_2^2Q_N
\\
&=:J_N+\mathcal{J}_{12}.
\end{aligned}
\]

We conclude with the following observation, whose proof is deferred until the end of the present section.

\begin{Lem}\label{L-CLT}
One has
\[
\mathcal{J}_{12}=\int\left\|\frac1N\sum_{k=1}^N\int_{\bU(d)}(K(\hat U,U_1)-K(U_k,U_1))\rho_f(t,d\hat U)\right\|_2^2\prod_{j=1}^N\rho_f(t,dU_j)\le\frac{16d}{N}.
\]
\end{Lem}

\end{proof}

\medskip
Our main result in this section is

\begin{Thm} \label{T3.1}  
Let $f^0\in\cP_2(\bU(d)\times\fsu(d))$ and let $f$ be the solution to \eqref{V-M} with initial data $f^0$. Let $F_N$ be the solution of the Liouville equation \eqref{Lio} with initial data $F_N^0=(f^0)^{\otimes N}$.

Then, the first marginal of $F_{N:1}$ of $F_N$ satisfies the inequality
\[  
\Dist_{MK,2}(f(t),F_{N:1}(t))^2\le\frac{8d}{5N}(e^{10t}-1), \qquad\text{ for all }t\ge 0.
\]
\end{Thm}

\begin{proof}  
It follows from \eqref{C-7} and Lemma \ref{L3.3} that $J_N$ satisfies
\[ 
\frac{dJ_N}{dt}\le 10J_N+\frac{16d}{N}, 
\]
for all $t\ge 0$. By Gronwall's lemma
\[
J_N(t)\le e^{10t}J_N(0)+\frac{8d}{5N}(e^{10t}-1),
\]
for all $t\ge 0$. Choosing $Q_N^0$ of the form 
\[
Q_N^0=\prod_{j=1}^Nf^0(dU_jdA_j)\de_{U_j}(V_j)\de_{A_j}(B_j)
\]
implies that $J_N(0)=0$. With this special choice of the initial coupling, the conclusion follows from Lemma \ref{L-JNMKd}.
\end{proof}

\medskip
\begin{proof}[Proof of Lemma \ref{L-CLT}]
Observe that
\[
\int_{\bU(d)}K(\hat U,U_1)\rho_f(t,d\hat U)-\frac1N\sum_{k=1}^NK(U_k,U_1)=\frac1N\sum_{k=1}^N\cV(U_k),
\]
with
\[
\cV(U_k):=\int_{\bU(d)}(K(\hat U,U_1)-K(U_k,U_1))\rho_f(t,d\hat U). 
\]
Then
\[  
\mathcal{J}_{12}=\int_{\bU(d)^N}\left\|\frac1N\sum_{k=1}^N\cV(U_k)\right\|_2^2\prod_{m=1}^N\rho_f(t,dU_m),
\]
since the integrand in $S$ depends only on the variables $U_1,\ldots,U_N$, while $Q_N(t)$ is belongs to $\Pi(f(t)^{\otimes N},F_N(t))$. Expanding the square in the integrand of $S$
\begin{align*}
\begin{aligned}
&\left\|\frac1N\sum_{k=1}^N\cV(U_k)\right\|_2^2 \\
& \hspace{0.5cm} =\frac1{N^2}\sum_{k=1}^N\|\cV(U_k)\|_2^2 +\frac{1}{N^2}\sum_{1\le j<k\le N}\Tr(\cV(U_j)^*\cV(U_k)+\cV(U_k)^*\cV(U_j)). 
\end{aligned}
\end{align*}
Now, if $U,V\in\bU(d)$, one has
\[
\|UV^*\|_2=\Tr(VU^*UV^*)^{1/2}=\sqrt{d},
\]
so that
\[
\begin{aligned}  
\|\cV(U_k)\|_2\le&\int_{\bU(d)}\|K(\hat U,U_1)-K(U_k,U_1)\|_2\rho_f(t,d\hat U)
\\
=&\int_{\bU(d)}\|(\hat U-U_k)^*U_1-U_1^*(\hat U-U_k)\|_2\rho_f(t,d\hat U)\le 4\sqrt{d}. 
\end{aligned}\]
This yields
\[
\left\|\frac1N\sum_{k=1}^N\cV(U_k)\right\|^2\le\frac{16d}{N}+\frac{1}{N^2}\sum_{1\le j<k\le N}\Tr(\cV(U_j)^*\cV(U_k)+\cV(U_k)^*\cV(U_j)). 
\]
On the other hand, if $1\le j<k\le N$, we have
\[
\begin{aligned}
\int_{\bU(d)^N}&\Tr(\cV(U_j)^*\cV(U_k))\prod_{m=1}^N\rho_f(t,dU_m)
\\
&= \Tr\left(\int_{\bU(d)^{N-1}}\cV(U_j)^*\prod_{m=1\atop m\not=k}^N\rho_f(t,dU_m)\int_{\bU(d)}\cV(U_k)\rho_f(t,dU_k)\right)=0,
\end{aligned}
\]
since, for each $k=2,\ldots,N$, 
\[
\int_{\bU(d)}\cV(U_k)\rho_f(t,dU_k)=\iint_{\bU(d)^2}(K(\hat U,U_1)-K(U_k,U_1))\rho_f(t,d\hat U)\rho_f(t,dU_k)=0. 
\]
Likewise, for $1\le j<k\le N$
\begin{align*}
\begin{aligned}
& \int_{\bU(d)^N}\Tr(\cV(U_k)^*\cV(U_j))\prod_{m=1}^N\rho_f(t,dU_m)
\\
& \hspace{1cm} =\int_{\bU(d)^N}\overline{\Tr(\cV(U_j)^*\cV(U_k))}\prod_{m=1}^N\rho_f(t,dU_m)=0,
\end{aligned}
\end{align*}
and this concludes the proof.
\end{proof}


\section{Emergent Dynamics: Identical Hamiltonians} \label{sec:4}


In this section, we establish the emergent dynamics of the kinetic Lohe equation \eqref{K-Lohe} in the case of equal Hamiltonians:
\[ 
H_1 = \cdots = H_N =: H=H^*\in M_d(\bbc). 
\]

Without loss of generality, as shown by the lemma below, we may assume that $H = 0$ and that $f$ is of the form
\[
f(t,dUdA) = \rho(t, U)\mu(dU)\delta_0(dA),\qquad\rho(t, U) := \rho_f(t, U),
\]
where $\rho$ is a time-dependent probability density on $\bU(d)$. 

If $f(t)$ is of this form and satisfies \eqref{V-M}, the local mass density $\rho$ satisfies
\begin{equation} \label{KL-I}
\displaystyle \d_t \rho + \frac{\kappa}{2}\Div_U\left(\rho \int_{\bU(d)} K(V,U) U \rho(t,V)\mu(dV)\right)=0, \quad U \in \bU(d),~t > 0,
\end{equation}
where we recall that $K(V, U) = VU^* -U V^*$. This is a scalar conservation law (with nonlocal flux function).

\medskip
If $H=iA_0\not=0$, seek $f$ of the form
\begin{equation}\label{fDeA0}
f(t,dUdA) = \l(t,e^{-tA_0}U)\mu(dU)\delta_{A_0}(dA),\qquad\rho_f(t,U) := \l(t,e^{-tA_0}U).
\end{equation}

\begin{Lem}
If $f(t)$ is a time-dependent probability measure of the form \eqref{fDeA0} which is a solution to \eqref{V-M}, then the time dependent probability density $\l(t,\cdot)$ is a solution to \eqref{KL-I}.
\end{Lem}

This is the only place in the present paper where we use the fact that the measure $\mu$ is translation invariant on $\bU(d)$. Henceforth, we use the following notation: for each $m\in\cP(\bU(d))$ and each $\phi\in C(\bU(d))$, we set
\[
\la\phi m\ra:=\int_{\bU(d)}\phi(V)m(dV).
\]

\begin{proof}
The time-dependent probability measure $f$ satisfies \eqref{K-Lohe} if and only if
\begin{align}\label{WeakK-Lohe}
\begin{aligned}
& \frac{d}{dt}\int_{\bU(d)\times\fsu(d)}\phi(U,A)f(t,dUdA)
\\
& \hspace{0.5cm} =\int_{\bU(d)\times\fsu(d)}(d_U\phi(U,A),(A+\la K(\cdot,U)\rho(t,\cdot)\mu\ra)U)f(t,dUdA).
\end{aligned}
\end{align}
Thus
\[
\la K(\cdot,U)\rho(t,\cdot)\mu\ra:=\int_{\bU(d)}K(\hat U,U)\rho(t,\hat U)\mu(d\hat U) =\int_{\bU(d)}K(\hat U,U)\l(t,e^{-tA_0}\hat U)\mu(d\hat U).\]
In this integral, substitute $\hat V=e^{-tA_0}\hat U$; since $\mu$ is translation invariant on $\bU(d)$ 
\[
\int_{\bU(d)}K(\hat U,U)\l(t,e^{-tA_0}\hat U)\mu(d\hat U)=\int_{\bU(d)}K(e^{tA_0}\hat V,U)\l(t,\hat V)\mu(d\hat V).
\]
On the other hand, $K(\hat U,U)=\hat UU^*-U\hat U^*$ satisfies the identity
\[
K(e^{tA_0}\hat V,e^{tA_0}V)=e^{tA_0}K(\hat V,V^*)e^{-tA_0},
\]
so that
\[
\la K(\cdot,e^{tA_0}V)\rho(t,\cdot)\mu\ra=e^{tA_0}\la K(\cdot,V)\l(t,\cdot)\mu\ra e^{-tA_0}.
\]
Observe that
\begin{align*}
\begin{aligned}
& \int_{\bU(d)\times\fsu(d)}\phi(U,A)f(t,dUdA) \\
& \hspace{0.5cm} =\int_{\bU(d)}\phi(U,A_0)\l(t,e^{-tA_0}U)\mu(dU)
=\int_{\bU(d)}\phi(e^{tA_0}V,A_0)\l(t,V)\mu(dV),
\end{aligned}
\end{align*}
and that
\begin{align*}
\begin{aligned}
& \int_{\bU(d)\times\fsu(d)}(d_U\phi(U,A),(A+\la K(\cdot,U)\rho(t,\cdot)\mu\ra)U)f(t,dUdA)
\\
& \hspace{0.2cm} =\!\!\int_{\bU(d)}(d_U\phi(U,A_0),(A_0+e^{tA_0}\la K(\cdot,e^{-tA_0}U)\l(t,\cdot)\mu\ra e^{-tA_0})U)\l(t,e^{-tA_0}U)\mu(dU)
\\
& \hspace{0.2cm}  =\int_{\bU(d)}(d_U\phi(e^{tA_0}V,A_0),(A_0e^{tA_0}+e^{tA_0}\la K(\cdot,V)\l(t,\cdot)\mu\ra V)\l(t,V)\mu(dV)
\end{aligned}
\end{align*}
by substituting $V=e^{-tA_0}U$ and using the translation invariance of $\mu$. Thus $f$ satisfies \eqref{WeakK-Lohe} if and only if
\begin{align*}
\begin{aligned}
& \frac{d}{dt}\int_{\bU(d)}\phi(e^{tA_0}V,A_0)\l(t,V)\mu(dV) \\
& \hspace{0.2cm}  =\!\!\int_{\bU(d)}\phi(e^{tA_0}V,A_0)\d_t\l(t,V)\mu(dV)\! +\!\!\int_{\bU(d)}(d_U\phi(e^{tA_0}V,A_0),Ae^{tA_0}V)\l(t,V)\mu(dV)
\\
& \hspace{0.2cm} =\!\!\int_{\bU(d)}(d_U\phi(e^{tA_0}V,A_0),(A_0e^{tA_0}+e^{tA_0}\la K(\cdot,V)\l(t,\cdot)\mu\ra V)\l(t,V)\mu(dV).
\end{aligned}
\end{align*}
This equality is recast as
\begin{align*}
\begin{aligned}
& \int_{\bU(d)}\phi(e^{tA_0}V,A_0)\d_t\l(t,V)\mu(dV)
\\
&  \hspace{0.5cm} =\int_{\bU(d)}(d_U\phi(e^{tA_0}V,A_0),e^{tA_0}\la K(\cdot,V)\l(t,\cdot)\mu\ra V)\l(t,V)\mu(dV)
\\
& \hspace{0.5cm}  =\int_{\bU(d)}(d_V\phi(e^{tA_0}V,A_0),\la K(\cdot,V)\l(t,\cdot)\mu\ra V)\l(t,V)\mu(dV),
\end{aligned}
\end{align*}
since $d_U\phi(e^{tA_0}V,A_0)=e^{tA_0}d_V\phi(e^{tA_0}V,A_0)$ and $e^{tA_0}$ is unitary. This last identity is equivalent to the fact that $\l(t,V)$ is a solution to \eqref{KL-I}.
\end{proof}

\medskip
Henceforth, we focus our attention to the equation \eqref{KL-I}, which drives the mean-field Lohe dynamics in the equal Hamiltonian case. We begin with an elementary conservation property satisfied by weak solutions to \eqref{KL-I}.

\begin{Lem} 
Let $\rho\equiv\rho(t,U)$ be a weak solution to \eqref{KL-I} such that $t\mapsto\rho(t,U)\mu(dU)$ is continuous on $[0,+\infty)$ with values in $\cP(\bU(d))$ equipped with the weak topology. Then
\[  
\frac{d}{dt}\int_{\bU(d)}\rho(t, U)\mu(dU)=0. 
\]
\end{Lem}

\begin{proof} The weak formulation of \eqref{KL-I} is
\begin{align*}
\begin{aligned}
& \frac{d}{dt}\int_{\bU(d)}\rho(t,U)\phi(U)\mu(dU)=\frac{\kappa}2\iint_{\bU(d)^2}(d\phi(U),K(V,U)U)\rho(t,V)\rho(t,U)\mu(dV)\mu(dU)
\end{aligned}
\end{align*}
which holds in the sense of distributions on $(0,+\infty)$ for all $\phi\in C^1(\bU(d))$. Specializing this identity to the test function $\phi(U)=1$ gives the announced result.
\end{proof}

\subsection{Particle path} \label{sec:4.1} 


Let $t\mapsto\rho(t)$ be continuous on $[0,+\infty)$ with values in the set of Borel probability measures on $\bU(d)$ equipped with the weak topology. Consider the differential equation
\[
\begin{aligned}
{}&\frac{dU(t)}{dt} = \frac{\kappa}{2}\int_{\bU(d)} K(V,U(t)) U(t) \rho(t,V)\mu(dV), \quad t > 0,
\\
&U(s)=U_s\in\bU(d).
\end{aligned}
\]
Since the map
\[
(t,U)\mapsto\int_{\bU(d)}K(V,U)U\rho(t,V)\mu(dV)
\]
is continuous on $[0,+\infty)\times\bU(d)$ and Lipschitz continuous in $U$ uniformly in $t\ge 0$, this differential equation generates a global flow denoted by $\mathcal{U}(t,s)$ on $\bU(d)$. In other words,
\[
t\mapsto\mathcal{U}(t,s)U_s
\]
is the solution of the Cauchy problem above. By formula (8.1.20) in \cite{AmbrosioGigliSavare}, the solution of \eqref{KL-I} satisfies
\[
\rho(t)\mu=\mathcal{U}(t,0)\#\rho(0)\mu=(\rho(0,\cdot)\circ\mathcal{U}(t,0)^{-1})\mathcal{U}(t,0)\#\mu.
\]
In particular
\[
\Supp(\rho(t)\mu)=\Supp(\rho(0,\cdot)\circ\mathcal{U}(t,0)^{-1})
\]
since $\mathcal{U}(t,s)$ is an homeomorphism for each $t,s\ge 0$ and $\mu(\Omega)>0$ for each nonempty open set of $\bU(d)$ (we recall that $\mu$ is the normalized Haar measure on $\bU(d)$).

\subsection{Complete synchronization} \label{sec:4.2}


Synchronization occurs in a solution of the Lohe matrix model $t\mapsto (U_1(t),\dots,U_N(t))$ when $\|U_j(t)-U_k(t)\|_2\to 0$ as $t\to\infty$ for all $j,k=1,\ldots,N$. This suggests the following definition.

\begin{Def}
Complete synchronization occurs in a solution $\rho$ of \eqref{KL-I} if
\[
\Diam(\Supp(\rho(t,\cdot)))\to 0\quad\hbox{ as }t\to+\infty.
\]
\end{Def}

\begin{Thm}\label{T-CSynchro}
Let $\rho\equiv\rho(t,U)$ be a weak solution of \eqref{KL-I} that is continuous in time with values in the set of probability densities equipped with the weak topology. Define
\[
D(t):=\sup\{\|V-W\|_2\hbox{ s.t. }\rho(0,\mathcal{U}(t,0)^{-1}V)\rho(0,\mathcal{U}(t,0)^{-1}W)>0\}.
\]
Assume that 
\[
\kappa>0\quad\hbox{Êand }\quad 0\le D(0)<\sqrt{2}.
\]
Then complete synchronization occurs in the solution $\rho$, and one has
\[
D(0)^2e^{-2\kappa t}(1-\tfrac12D(0)^2(1-e^{-2\kappa t}))\le D(t)^2\le\frac{2D(0)^2}{(2-D(0)^2)e^{2\kappa t}+D(0)^2}
\]
for all $t\ge 0$.
\end{Thm}

\smallskip
We begin with the following auxiliary computation, which will be systematically used later, including in the proof of Theorem \ref{T-CSynchro}. 

\begin{Lem}\label{L-AuxComp}
Let $U_1,U_2\in\bU(d)$ and let $m\in\cP(\bU(d))$. Then one has
\begin{align*}
\begin{aligned}
& \Tr((U_1-U_2)^*(U_2\la V^*m\ra U_2-U_1\la V^*m\ra U_1)) \\
& \hspace{0.5cm} +\Tr((U_2^*\la V m\ra U_2^*-U_1^*\la Vm\ra U_1^*)(U_1-U_2))
\\
& \hspace{0.5cm} =-2\|U_1-U_2\|_2^2+\tfrac12\Tr(\la(V-U_2)(V-U_2)^*m\ra(U_1-U_2)(U_1-U_2)^*)
\\
&\hspace{0.5cm} +\tfrac12\Tr(\la(V-U_1)(V-U_1)^*m\ra(U_1-U_2)(U_1-U_2)^*).
\end{aligned}
\end{align*}
\end{Lem}

\begin{proof}
Call
\[
\begin{aligned}
S:=&\Tr((U_1-U_2)^*(U_2\la V^*m\ra U_2-U_1\la V^*m\ra U_1)
\\
&+(U_2^*\la Vm\ra U_2^*-U_1^*\la Vm\ra U_1^*)(U_1-U_2)).
\end{aligned}
\]
Observe that
\begin{align*}
\begin{aligned}
& \la(V-U_1)(V-U_1)^*m\ra(U_1-U_2)(U_1-U_2)^* \\
& \hspace{0.5cm} =(2I-U_1\la V^*m\ra-\la Vm\ra U_1^*)(2I-U_2U_1^*-U_1U_2^*),
\end{aligned}
\end{align*}
and
\begin{align*}
\begin{aligned}
& \la(V-U_2)(V-U_2)^*m\ra(U_1-U_2)(U_1-U_2)^*
\\
& \hspace{0.5cm} =(2I-U_2\la V^*m\ra-\la Vm\ra U_2^*)(2I-U_2U_1^*-U_1U_2^*).
\end{aligned}
\end{align*}
Therefore
\begin{align*}
\begin{aligned}
& \la(V-U_2)(V-U_2)^*m\ra(U_1-U_2)(U_1-U_2)^* \\
& +\la(V-U_1)(V-U_1)^*m\ra(U_1-U_2)(U_1-U_2)^* =4(U_1-U_2)(U_1-U_2)^*-R,
\end{aligned}
\end{align*}
where
\[
R:=((U_1+U_2)\la V^*m\ra+\la Vm\ra(U_1+U_2)^*)(2I-U_2U_1^*-U_1U_2^*).
\]
Then
\begin{align*}
\begin{aligned}
\Tr(R)&=2\Tr(\la V^*m\ra(U_1+U_2)+(U^*_1+U^*_2)\la Vm\ra)
\\
&-\Tr((U_1+U_2)\la V^*m\ra(U_2U_1^*+U_1U_2^*)) \\
& -\Tr((U_2U_1^*+U_1U_2^*)\la Vm\ra(U_1+U_2)^*).
\end{aligned}
\end{align*}
Observe that
\[
\begin{aligned}
2\Tr(\la V^*m\ra U_1)-\Tr(U_2\la V^*m\ra U_1U_2^*)&=\Tr(\la V^*m\ra U_1)
\\
2\Tr(\la V^*m\ra U_2)-\Tr(U_1\la V^*m\ra U_2U_1^*)&=\Tr(\la V^*m\ra U_2)
\\
2\Tr(U^*_1\la Vm\ra)-\Tr(U_2U_1^*\la Vm\ra U_2^*)&=\Tr(U^*_1\la Vm\ra)
\\
2\Tr(U^*_2\la Vm\ra)-\Tr(U_1U_2^*\la Vm\ra U_1^*)&=\Tr(U^*_2\la Vm\ra),
\end{aligned}
\]
so that
\[
\begin{aligned}
\Tr(R)=&\Tr(\la V^*m\ra(U_1+U_2)+(U^*_1+U^*_2)\la Vm\ra)
\\
&-\Tr(U_2\la V^*m\ra U_2U_1^*+U_1\la V^*m\ra U_1U_2^*)
\\
&-\Tr(U_2U_1^*\la Vm\ra U^*_1+U_1U_2^*\la Vm\ra U_2^*)=-S.
\end{aligned}
\]
\end{proof}

\smallskip
After this preliminaries, we give the proof of the main result in this section.

\begin{proof}[Proof of Theorem \ref{T-CSynchro}] The proof of this result is rather long, and split in several steps. \newline

\noindent $\bullet$~\textit{Step 1.} Let $U_1^0,U_1^2\in\bU(d)$, and set
\[
U_1(t):=\mathcal{U}(t,0)U_1^0\,,\qquad U_2(t):=\mathcal{U}(t,0)U_2^0.
\]
Then
\[
\begin{aligned}
\dot{U}_1(t)=\tfrac{\kappa}2(\la V\rho(t,\cdot)\mu\ra-U_1(t)\la V^*\rho(t,\cdot)\mu\ra U_1(t)),
\\
\dot{U}_2(t)=\tfrac{\kappa}2(\la V\rho(t,\cdot)\mu\ra-U_2(t)\la V^*\rho(t,\cdot)\mu\ra U_2(t)),
\end{aligned}
\]
so that
\[
\frac{d}{dt}(U_1-U_2)(t)=-\tfrac{\kappa}2(U_1(t)\la V^*\rho(t,\cdot)\mu\ra U_1(t)-U_2(t)\la V^*\rho(t,\cdot)\mu\ra U_2(t)).
\]
Hence
\[
\begin{aligned}
\frac{d}{dt}\|U_1-U_2\|_2^2=&-\tfrac{\kappa}2\Tr((U^*_1-U^*_2)(U_1\la V^*\rho\mu\ra U_1-U_2\la V^*\rho\mu\ra U_2))
\\
&-\tfrac{\kappa}2\Tr((U^*_1\la V\rho\mu\ra U^*_1-U^*_2\la V\rho\mu\ra U^*_2)(U_1-U_2)).
\end{aligned}
\]

\smallskip
At this point, we apply Lemma \ref{L-AuxComp} to the time dependent probability measure $m(t):=\rho(t,\cdot)\mu$: one finds that
\begin{align*}
\begin{aligned}
\frac{d}{dt}\|U_1-U_2\|_2^2&=-2\kappa\|U_1-U_2\|_2^2 +\tfrac{\kappa}2\Tr(\la(V-U_2)(V-U_2)^*\rho\mu\ra(U_1-U_2)(U_1-U_2)^*)
\\
&+\tfrac{\kappa}2\Tr(\la(V-U_1)(V-U_1)^*\rho\mu\ra(U_1-U_2)(U_1-U_2)^*).
\end{aligned}
\end{align*}
Assume that 
\[
\rho(0,U_1^0)\rho(0,U_2^0)>0.
\]
Set 
\[
d(t):=\sup\{\|V-W\|\hbox{ s.t. }\rho(0,\mathcal{U}(t,0)^{-1}V)\rho(0,\mathcal{U}(t,0)^{-1}W)>0\},\qquad t\ge 0\,.
\]
Notice the difference between $D(t)$, which is defined in terms of the Frobenius norm $\|\cdot\|_2$, and $d(t)$, defined in terms of the operator norm $\|\cdot\|$. 

Since 
\begin{align*}
\begin{aligned}
& \la(V-U_2(t))(V-U_2(t))^*\rho(t,\cdot)\mu\ra \\
& \hspace{0.5cm} =\int_{\bU(d)}(V-U_2(t))(V-U_2(t))^*\rho(t,V)\mu(dV)
\\
& \hspace{0.5cm}=\int_{\Supp(\rho(0,\cdot))}(\mathcal{U}(t,0)V^0-\mathcal{U}(t,0)U^0_2)(\mathcal{U}(t,0)V^0-\mathcal{U}(t,0)U^0_2)^*\rho(0,V^0)\mu(dV^0),
\end{aligned}
\end{align*}
one has
\[
\|\la(V-U_2(t))(V-U_2(t))^*\rho(t,\cdot)\mu\ra\|\le\sup_{\rho(0,V^0)>0}\|\mathcal{U}(t,0)V^0-\mathcal{U}(t,0)U^0_2\|^2\le d(t)^2.
\]
Moreover, since $\la(V-U_2)(V-U_2)^*\rho(t,\cdot)\mu\ra$ is a nonnegative self-adjoint matrix, one has
\begin{align}\label{Tr<>U1U2}
\begin{aligned}
& \Tr(\la(V-U_2(t))(V-U_2(t))^*\rho(t,\cdot)\mu\ra(U_1(t)-U_2(t))(U_1(t)-U_2(t))^*)
\\
& \hspace{0.5cm} =\|(U_1(t)-U_2(t))^*\la(V-U_2(t))(V-U_2(t))^*\rho(t,\cdot)\mu\ra^{1/2}\|_2^2
\\
&\hspace{0.5cm} \le\|U_1(t)^*-U_2(t)^*\|_2^2\|\la(V-U_2(t))(V-U_2(t))^*\rho(t,\cdot)\mu\ra^{1/2}\|^2
\\
& \hspace{0.5cm}=\|U_1(t)-U_2(t)\|_2^2\|\la(V-U_2(t))(V-U_2(t))^*\rho(t,\cdot)\mu\ra\| \\
& \hspace{0.5cm} \le d(t)^2\|U_1(t)-U_2(t)\|_2^2
\end{aligned}
\end{align}
according to the second inequality in \eqref{IneqNorms}. Exchanging $U_1$ and $U_2$ leads to the same bound for the term
\[
\Tr(\la(V-U_1(t))(V-U_1(t))^*\rho(t,\cdot)\mu\ra(U_1(t)-U_2(t))(U_1(t)-U_2(t))^*).
\]
Hence, we conclude from the formula for $\frac{d}{dt}\|U_1-U_2\|_2^2$ given above that
\begin{align*}
\begin{aligned}
& -\kappa(2+d(t)^2)\|U_1(t)-U_2(t)\|_2^2 \\
& \hspace{1.5cm}  \le \frac{d}{dt}\|U_1(t)-U_2(t)\|_2^2 \le-\kappa(2-d(t)^2)\|U_1(t)-U_2(t)\|_2^2.
\end{aligned}
\end{align*}
According to the first inequality in \eqref{IneqNorms}, this implies that
\begin{align*}
\begin{aligned}
&-\kappa(2+D(t)^2)\|U_1(t)-U_2(t)\|_2^2 \\
& \hspace{1.5cm} \le\frac{d}{dt}\|U_1(t)-U_2(t)\|_2^2 \le -\kappa(2-D(t)^2)\|U_1(t)-U_2(t)\|_2^2.
\end{aligned}
\end{align*}

\noindent $\bullet$~\textit{Step 2.} Applying Lemma \ref{L-Barrier} in Appendix \ref{App:B} to the function $x(t)=\|U_1-U_2\|_2^2(t)$ with $\Dlt(t)=D(t)$ and $\a=0$, we conclude that 
\[
D(t)\le y(t)\qquad\hbox{ for all }t\ge 0,
\]
where $y$ is the solution of the Cauchy problem
\[
\dot y(t)=-\ka y(t)+\tfrac12\ka y(t)^3,\quad y(0)=D(0).
\]
An explicit computation shows that
\[
y(t)^2=\frac{2D(0)^2}{(2-D(0)^2)e^{2\kappa t}+D(0)^2}
\]
which implies the upper bound in Theorem \ref{T-CSynchro}. \newline

\noindent $\bullet$~\textit{Step 3.} As for the lower bound, observe that
\[
\frac{d}{dt}\|U_1-U_2\|_2^2(t)\ge-\ka(2+D(t)^2)\|U_1-U_2\|_2^2(t)\ge-\ka(2+y(t)^2)\|U_1-U_2\|_2^2(t),
\]
so that
\[
\|U_1-U_2\|_2^2(t)\ge\|U_1-U_2\|_2^2(0)\exp\left(-\ka\int_0^t(2+y(s)^2)ds\right)
\]
for all $t\ge 0$. Since $\Supp(\rho^0)$ is compact, there exists $U_1^0,U_2^0\in\bU(d)$ such that
\[
\|U_1^0-U_2^0\|_2^2=D(0)^2.
\]
Choosing $U_1(0)=U_1^0$ and $U_2(0)=U_2^0$, we conclude that
\begin{align*}
\begin{aligned}
D(t)^2 &\ge\|U_1-U_2\|_2^2(t) \\
&\ge D(0)^2\exp\left(-\ka\int_0^t(2+y(s)^2)ds\right) \ge D(0)^2e^{-2\kappa t}(1-\tfrac12D(0)^2(1-e^{-2\kappa t}))
\end{aligned}
\end{align*}
for all $t\ge 0$, which is the announced lower bound.

\end{proof}


\section{Emergent Dynamics: Nonidentical Hamiltonians} \label{sec:5}


In this section, we discuss the emergent dynamics for the kinetic Lohe equation \eqref{K-Lohe} in the case of nonidentical hamiltonians. 

We recall the kinetic Lohe equation with unknown the time-dependent probability density $f\equiv f(t,U,A)$ on $\bU(d)\times\fsu(d)$, in the form
\begin{equation}\label{F-1}
\d_tf+\Div_{U}(fAU)+\frac{\kappa}{2}\Div_U\left(f\int_{\bU(d)\times\fsu(d)}K(V,U)Uf(t,V, B)\mu(dV)\gamma(dB)\right)=0.
\end{equation}
Set
\[
\rho_f(t,U):=\int_{\fsu(d)}f(t,U,A)\g(dA),
\]
together with
\[
D(t):=\sup\{\|U_1-U_2\|_2\hbox{ s.t. }U_1\hbox{ and }U_2\in\Supp(\rho_f(t,\cdot))\},
\]
and
\[
\a:=\sup\{\|A_1-A_2\|_2\hbox{ s.t. }(U_1,A_1)\hbox{ and }(U_2,A_2)\in\Supp(f(t,\cdot,\cdot))\}.
\]
Since the $A$-component of vector field in \eqref{F-1} is identically $0$, the quantity in the right hand side of this equality is obviously independent of $t$. We recall the notation introduced in Section \ref{sec:4}: 
\[
\la\phi m\ra:=\int_{\bU(d)}\phi(V)m(dV)
\]
for each $m\in\cP(\bU(d))$ and each $\phi\in C(\bU(d))$.

Assume that $t\mapsto f(t,U,A)\mu(dU)\g(dA)$ be continuous on $[0,+\infty)$ with values in $\cP(\bU(d)\times\fsu(d))$ equipped with the weak topology. Consider the differential system
\[
\begin{aligned}
{}&\dot U(t) = AU+\frac{\kappa}{2}\int_{\bU(d)} K(V,U(t))U(t)\rho_f(t,V)\mu(dV),\quad\dot A(t)=0,\quad t>0,
\\
&(U,A)(s)=(U_s,A_s)\in\bU(d)\times\fsu(d).
\end{aligned}
\]
The same argument as in Section \ref{sec:4.1} proves the existence of a flow $\Phi(t,s)$ on $\bU(d)\times\fsu(d)$ such that
\[
t\mapsto\Phi(t,s)(U_s,A_s)
\]
is the unique solution to the Cauchy problem above. Moreover, $\Phi(t,s)$ is of the form
\[
\Phi(t,s)(U_s,A_s)=(\mathcal{U}[A_s](t,s)U_s,A_s),
\]
where $\mathcal{U}[A](t,s)$ is a flow on $\bU(d)$ for each $A\in\fsu(d)$, and
\[
f(t,\cdot,\cdot)\mu\otimes\g=\Phi(t,0)\#(f(0,\cdot,\cdot)\mu\otimes\g)=(f(0,\cdot,\cdot)\circ\Phi(t,0)^{-1})\Phi(t,0)\#(\mu\otimes\g).
\]
In particular
\[
\Supp(f(t,\cdot,\cdot)\mu\otimes\g)=\Supp(f(0,\cdot,\cdot)\circ\Phi(t,0)^{-1})
\]
since $\g(\Om)>0$ for each nonempty open $\Om\subset\fsu(d)$.

\subsection{An invariant region} \label{sec:5.1}


A key step in the problem of synchronization is the following a priori estimate.

\begin{Lem}\label{L-InvReg}
Let $f^0$ be a compactly supported probability density on $\bU(d)\times\fsu(d)$ such that 
\[
0<\a<(2/3)^{3/2}\ka\quad\hbox{ and }\quad 0<D(0)<\zeta_2(\a/\ka),
\]
where $\zeta_2(\eta)$ designates the largest root of $\tfrac12X^3-X+\eta$ for all $\eta\in[0,(2/3)^{3/2})$. Then, the solution of \eqref{F-1} with initial data $f^0$ satisfies
\[
D(t)<\zeta_2(\a/\ka)\quad\hbox{ for all }t\ge 0.
\]
Moreover, there exists $T\equiv T[\a,\ka,D^0]>0$ such that
\[
t\ge T[\a,\ka,D^0]\implies D(t)\le\sqrt{2/3}.
\]
\end{Lem}

\begin{proof}
Pick $(U_1^0,A_1)$ and $(U_2^0,A_2)$ such that $f^0(U_1^0,A_1)f^0(U_2^0,A_2)>0$, and set
\[
U_1(t)=\mathcal{U}[A_1](t,0)U_1^0,\qquad U_2(t)=\mathcal{U}[A_2](t,0)U_2^0.
\]
Then
\[
\begin{aligned}
\dot U_1=A_1U_1+\tfrac\kappa{2}(\la V\rho_f\mu\ra-U_1\la V^*\rho_f\mu\ra U_1),
\\
\dot U_2=A_2U_2+\tfrac\kappa{2}(\la V\rho_f\mu\ra-U_2\la V^*\rho_f\mu\ra U_2),
\end{aligned}
\]
so that, observing that $\|U_1-U_2\|_2^2=\Tr(2I-U_1^*U_2-U_2^*U_1)$, one has
\begin{equation}\label{BoundDU1U2A}
\begin{aligned}
\frac{d}{dt}\|U_1-U_2\|_2^2=&-\Tr(U_1^*\dot U_2+U_2^*\dot U_1)-\Tr(\dot U_2^*U_1+\dot U_1^*U_2)
\\
=&-\Tr(U_1^*A_2U_2+U_2^*A_1U_1)-\Tr(U_2^*A_2^*U_1+U_1^*A_1^*U_2)
\\
&-\tfrac{\kappa}2\Tr(U_1^*(\la V\rho_f\mu\ra-U_2\la V^*\rho_f\mu\ra U_2))
\\
&-\tfrac{\kappa}2\Tr(U_2^*(\la V\rho_f\mu\ra-U_1\la V^*\rho_f\mu\ra U_1))
\\
&-\tfrac{\kappa}2\Tr((\la V^*\rho_f\mu\ra-U_2^*\la V\rho_f\mu\ra U_2^*)U_1)
\\
&-\tfrac{\kappa}2\Tr((\la V^*\rho_f\mu\ra-U_1^*\la V\rho_f\mu\ra U_1^*)U_2).
\end{aligned}
\end{equation}
Since $A_1^*=-A_1$ and $A_2^*=-A_2$, the first two terms on the last right hand side are simplified as
\begin{align}\label{TrAU}
\begin{aligned}
&-\Tr(U_1^*A_2U_2+U_2^*A_1U_1) -\Tr(U_2^*A_2^*U_1+U_1^*A_1^*U_2)
\\
& \hspace{1cm} =\Tr((A_2-A_1)(U_1U_2^*-U_2U_1^*))
\\
& \hspace{1cm}=\Tr((A_2-A_1)(U_1U_2^*-I)) +\Tr((A_2-A_1)(I-U_2U_1^*))
\\
&\hspace{1cm} \le \|A_2-A_1\|_2\|U_1(U_2^*-U_1^*)\|_2 +\|A_2-A_1\|_2\|(U_1-U_2)U_1^*\|_2
\\
&\hspace{1cm} \le2\|A_2-A_1\|_2\|U_1-U_2\|_2.
\end{aligned}
\end{align}
In the chain of inequalities above, the penultimate inequality is the Cauchy-Schwarz inequality \eqref{CSFrob} for the Frobenius norm, while the final bound is the second inequality in \eqref{IneqNorms}.

The last two terms on the right hand side of \eqref{BoundDU1U2A} are mastered by applying Lemma \ref{L-AuxComp}:
\[
\begin{aligned}
\frac{d}{dt}\|U_1-U_2\|_2^2=&\Tr((A_2-A_1)(U_1U_2^*-U_2U_1^*)-2\kappa\|U_1-U_2\|_2^2
\\
&+\tfrac{\kappa}2\Tr(\la(V-U_2)(V-U_2)^*\rho_f\mu\ra(U_1-U_2)(U_1-U_2)^*)
\\
&+\tfrac{\kappa}2\Tr(\la(V-U_1)(V-U_1)^*\rho_f\mu\ra(U_1-U_2)(U_1-U_2)^*).
\end{aligned}
\]
Thus, arguing as in \eqref{Tr<>U1U2} and using the first inequality in \eqref{IneqNorms} leads to the following upper bound:
\begin{align*}
\begin{aligned}
& \frac{d}{dt}\|U_1(t)-U_2(t)\|_2^2 \\
&\hspace{0.5cm} \le2\|A_2(t)-A_1(t)\|_2\|U_1(t)-U_2(t)\|_2 -2\kappa\|U_1(t)-U_2(t)\|_2^2 \\
& \hspace{0.7cm}+\kappa D(t)^2\|U_1(t)-U_2(t)\|_2^2(t),
\end{aligned}
\end{align*}
or equivalently
\begin{align*}
\begin{aligned}
&\frac{d}{dt}\|U_1(t)-U_2(t)\|_2 \\
& \hspace{0.5cm} \le\|A_2(t)-A_1(t)\|_2 -\kappa\|U_1(t)-U_2(t)\|_2+\tfrac{\kappa}2D(t)^2\|U_1(t)-U_2(t)\|_2
\\
& \hspace{0.5cm} \le \a-\kappa(1-\tfrac12D(t)^2)\|U_1(t)-U_2(t)\|_2.
\end{aligned}
\end{align*}
Since we have assumed that $D^0\in(0,\zeta_2(\a/\ka))$, we conclude that
\[
D(t)\le y(t)<\zeta_2(\a/\ka)\quad\hbox{ for all }t\ge 0,
\]
where $y$ is the solution of the Cauchy problem
\[
\dot y=\a-\ka y+\tfrac12\ka y^3,\qquad y(0)=D^0,
\]
by applying Lemma \ref{L-Barrier} in Appendix \ref{App:B} to $x(t):=\|U_1(t)-U_2(t)\|_2$. \newline

\noindent If $y(0)\in(\zeta_1(\a/\ka),\zeta_2(\a/\ka))$ the function $y$ is decreasing and converges to $\zeta_1(\a/\ka)$ as $t\to+\infty$. Hence there exists $T\equiv T[\a,\ka, D^0]\in(0,+\infty)$ such that
\[
t\ge T[\a,\ka, D^0]\implies y(t)\in(\zeta_1(\a/\ka),\sqrt{2/3}],
\]
so that
\[
t\ge T[\a,\ka, D^0]\implies 0\le D(t)\le\sqrt{2/3}.
\]
If $y(0)\in[0,\zeta_1(\a/\ka)]$, the function $y$ is nondecreasing and converges to $\zeta_1(\a/\ka)$ as $t\to+\infty$. In particular
\[
t\ge 0\implies 0\le D(t)\le\zeta_1(\a/\ka)\le\sqrt{2/3}.
\]
\end{proof}

\subsection{Practical synchronization} \label{sec:5.2}


We first recall the following elementary lemma from measure theory.

\begin{Lem}
Let $t\mapsto\rho(t)$ be a continuous function on $[0,+\infty)$ with values in the set of Borel probability measures on $\bU(d)$ equipped with the weak topology. If
\[
\Diam(\Supp(\rho(t))\to0\quad\hbox{ as }t\to+\infty,
\]
then, for each function $\Phi\in C(\bU(d)\times\bU(d))$ such that
\[
U\not=V\implies\Phi(U,V)>0\quad\hbox{Êwhile }\Phi(U,U)=0\quad\hbox{ for all }U,V\in\bU(d),
\]
one has
\[
\iint_{\bU(d)\times\bU(d)}\Phi(U,V)\rho(t,dU)\rho(t,dV)\to 0\quad\hbox{Êas }t\to+\infty.
\]
\end{Lem}

\begin{proof}
For each $\eta>0$, there exists $T(\eta)>0$ such that $U,V\in\Supp(\rho(t))$ implies $\|U-V\|_2<\eta$ whenever $t>T(\eta)$. Equivalently $(U,V)\in\Supp(\rho(t)\otimes\rho(t))$ implies $\|U-V\|_2<\eta$ whenever $t>T(\eta)$. On the other hand, for 
each $\eps>0$, there exists $\eta(\eps)>0$ such that
\[
\|U-V\|_2<\eta(\eps)\implies\Phi(U,V)<\eps.
\]
Indeed, since $\Phi$ is continuous on the compact $\bU(d)\times\bU(d)$, it is uniformly continuous. Hence, for $t>T(\eta(\eps))$
\[
\iint_{\bU(d)\times\bU(d)}\Phi(U,V)\rho(t,dU)\rho(t,dV)\le\sup_{\|U-V\|_2<\eta(\eps)}\Phi(U,V)\le\eps,
\]
which is the desired conclusion.
\end{proof}

This lemma suggests the following weaker variant of the notion of complete synchronization introduced in Section \ref{sec:4}. For each solution $f$ of \eqref{F-1}, set
\[
\L[f](t):=\int_{(\bU(d)\times\fsu(d))^2}\|U-V\|_2^2f(t,U,A)f(t,V,B)\mu(dA)\g(dA)\mu(dV)\g(dB).
\]

\begin{Def}\label{D-PSynchro}
One says that a family $f_\ka$ of solutions of \eqref{F-1} with coupling strength $\ka$ exhibits ``practical synchronization'' if
\[
\lim_{\ka\to+\infty}\varlimsup_{t\to+\infty}\L[f_\ka](t)=0.
\]
\end{Def}

\medskip
Our main result in this section is the following theorem.

\begin{Thm}
Let $f^0$ be a compactly supported probability density on $\bU(d)\times\fsu(d)$, and set
\[
\begin{aligned}
D^0&:=\sup\{\|U-V\|_2\hbox{ s.t. }U\hbox{ and }V\in\Supp(\rho_{f^0})\},
\\
\a&:=\sup\{\|A-B\|_2\hbox{ s.t. }(U,A)\hbox{ and }(V,B)\in\Supp(f^0)\}.
\end{aligned}
\]
Assume that $\ka>(3/2)^{3/2}\a>0$ and that $0<D^0<\zeta_2(\a/\ka)$. Then the family $f_\ka$ of weak solutions of \eqref{F-1} with coupling strength $\ka$ exhibits practical synchronization.
\end{Thm}

\begin{proof}
Denote for simplicity $\L_\ka(t):=\L[f_\ka](t)$. Since $f$ is a weak solution of \eqref{F-1}
\begin{align*}
\begin{aligned}
& \dot\L_\ka\! =\!\int_{(\bU(d)\times\fsu(d))^2}\left(\Tr((U\!-\!V)^*(AU\!-\!BV\!-\!\tfrac{\ka}2(U\la W^*\rho_f\mu\ra U\!-\!V\la W^*\rho_f\mu\ra V))\right.
\\
& \hspace{0.5cm} \left.+\Tr((U^*A^*-V^*B^*-\tfrac{\ka}2(U^*\la W\rho_f\mu\ra U^*-V^*\la W\rho_f\mu\ra V^*)(U-V))\right)
\\
&  \hspace{0.5cm} \times f_\ka(t,U,A)f_\ka(t,V,B)\mu(dU)\g(dA)\mu(dV)\g(dB).
\end{aligned}
\end{align*}

Since $A^*=-A$ and $B^*=-B$ while $U,V\in\bU(d)$, one has
\[
\Tr((U-V)^*(AU-BV)+(U^*A^*-V^*B^*)(U-V)) \le 2\|A-B\|_2\|U-V\|_2
\]
according to \eqref{TrAU}. Likewise
\begin{align*}
\begin{aligned}
& -\tfrac{\ka}2\Tr((U-V)^*(U\la W^*\rho_f\mu\ra U-V\la W^*\rho_f\mu\ra V))
\\
& \hspace{0.5cm} -\tfrac{\ka}2\Tr((U^*\la W\rho_f\mu\ra U^*-V^*\la W\rho_f\mu\ra V^*)(U-V))
\\
& \hspace{0.5cm} =-2\ka\|U-V\|_2^2+\tfrac{\ka}2\Tr(\la(W-V)(W-V)^*\rho_f\mu\ra(U-V)(U-V)^*)
\\
& \hspace{0.5cm} +\tfrac{\ka}2\Tr(\la(W-U)(W-U)^*\rho_f\mu\ra(U-V)(U-V)^*)
\end{aligned}
\end{align*}
according to Lemma \ref{L-AuxComp}. Therefore
\begin{align*}
\begin{aligned}
& \dot\L_\ka(t)\le2\int_{(\bU(d)\times\fsu(d))^2}\|A-B\|_2\|U-V\|_2
\\
& \hspace{0.5cm} \times f_\ka(t,U,A)f_\ka(t,V,B)\mu(dU)\g(dA)\mu(dV)\g(dB)
\\
& \hspace{0.5cm} -2\ka\L_\ka(t)+\tfrac{\ka}2\int_{(\bU(d)\times\fsu(d))^2}\Tr(\la(W-V)(W-V)^*\rho_f\mu\ra(U-V)(U-V)^*)
\\
& \hspace{0.5cm} \times f_\ka(t,U,A)f_\ka(t,V,B)\mu(dU)\g(dA)\mu(dV)\g(dB)
\\
& \hspace{0.5cm} +\tfrac{\ka}2\int_{(\bU(d)\times\fsu(d))^2}\Tr(\la(W-U)(W-U)^*\rho_f\mu\ra(U-V)(U-V)^*)
\\
& \hspace{0.5cm}  \times f_\ka(t,U,A)f_\ka(t,V,B)\mu(dU)\g(dA)\mu(dV)\g(dB).
\end{aligned}
\end{align*}
Set
\[
\cA_\ka(t):=\int_{(\bU(d)\times\fsu(d))^2}\|A-B\|^2_2f_\ka(t,U,A)f_\ka(t,V,B)\mu(dU)\g(dA)\mu(dV)\g(dB).
\]
Since the vector field in \eqref{F-1} has zero $A$-component, 
\[
\cA_\ka(t)=\cA_\ka(0)\qquad\hbox{ for all }t\ge 0.
\]
By the Cauchy-Schwarz inequality and the bound \eqref{Tr<>U1U2}, one has
\[
\dot\L_\ka(t)\le 2\cA_\ka(0)^{1/2}\L_\ka(t)^{1/2}-2\ka\L_\ka(t)+\ka D(t)^2\L_\ka(t).
\]
By Lemma \ref{L-InvReg} 
\[
\begin{aligned}
\frac{d}{dt}\L_\ka(t)^{1/2}\le&\cA_\ka(0)^{1/2}-\ka\L_\ka(t)+\tfrac13\ka\L_\ka(t)\qquad\hbox{ for all }t\ge T[\a,\ka,D^0],
\\
\le
&\a-\tfrac23\ka\L_\ka(t).
\end{aligned}
\]
Hence
\[
\L_\ka(t)^{1/2}\le\L_\ka(T[\a,\ka,D^0])^{1/2}e^{-2\ka (t-T[\a,\ka,D^0])/3}+\frac{3\a}{2\ka}(1-e^{-2\ka (t-T[\a,\ka,D^0])/3}),
\]
so that
\[
\varlimsup_{t\to+\infty}\L_\ka(t)=\frac{3\a}{2\ka},
\]
which implies practical synchronization.
\end{proof}


\section{Conclusion} \label{sec:6}


In this paper, we presented a mean-field limit from the Lohe matrix model for quantum synchronization to the Vlasov type mean-field equation on the phase space $\bU(d) \times \fsu(d)$ and the emergent synchronization estimates for the 
derived mean-field model using the nonlinear functional approach. The Lohe matrix model was originally proposed as a nonabelian toy generalization of the Kuramoto phase model with an all-to-all coupling, and can be reduced to some
swarming model on a sphere in some special cases. 

In the absence of coupling, the Lohe matrix model is simply $N$ copies of a finite-dimensional Schr\"{o}dinger equation with constant hamiltonians. In the mean-field limit (with $N\to\infty$ and coupling rate of order $O(1/N)$), we use the 
BBGKY hierarchy with factorized $2$-oscillator distribution closure in order to derive a kinetic Lohe equation. This equation falls in the class of Vlasov type equations for the single-oscillator probability density function $f$ defined on the 
generalized phase space $\bU(d) \times \fsu(d)$. 

Moreover, using the analytical tools developed for the mean-field limit for the $N$-body Schr\"{o}dinger equation, we show that the single-oscillator marginal of the $N$-oscillator distribution approaches the solution to the kinetic equation
to within $O(1/\sqrt{N})$ in quadratic Monge-Kantorovich distance over any finite-time interval. 

We also provide two settings in terms of the coupling strength and initial data for the kinetic Lohe model leading to complete or practical synchronization in the cases of identical or nonidentical hamiltonians respectively. For this, we 
used the Lagrangian formulation of the kinetic equation and some appropriate Lyapunov functional measuring the degree of synchronization. In the present work, we study only the emergent dynamics of the kinetic Lohe model. 

There are still many interesting issues to be investigated, such as the existence of a critical coupling strength from disordered phase to ordered phase, the structure of the emergent phase-locked state, the existence of an optimal coupling 
strength for complete synchronization. These questions will be explored in future work.


\begin{appendix}

\section{Proof of Proposition \ref{P-LoheKinEq}}\label{App:A}


Set $E=\bU(d)\times\fsu(d)$. For each $T>0$, let $\cE_T:=C([0,T],(\cP_1(E),\MK))$, where $\MK$ is the Monge-Kantorovich (or Wasserstein) distance with exponent $1$ (also known as the Kantorovich-Rubinstein distance) whose definition 
is recalled below:
\[
\MK(P_1,P_2):=\inf_{m\in\Pi(P_1,P_2)}\int_{E\times E}(\|U_1-U_2\|_2+\|A_1-A_2\|_2)m(dU_1dA_1dU_2dA_2).
\]
If $\phi:\,E\to\bbc$ is Lipschitz continuous for the $\ell^1$ distance associated to the Frobenius norm, i.e.
\[
|\phi(U_1,A_1)-\phi(U_2,A_2)|\le\Lip(\phi)(\|U_1-U_2\|_2+\|A_1-A_2\|_2)
\]
for all $(U_1,A_1)$ and $(U_2,A_2)\in E$ (with $\Lip(\phi)<\infty$), then
\begin{align*}
\begin{aligned}
& \left|\int_E\phi(V,B)P_1(dVdB))-\int_E\phi(V,B)P_2(dVdB)\right|
\\
& \hspace{0.5cm} \le\Lip(\phi)\int_{E\times E}(\|U_1-U_2\|_2+\|A_1-A_2\|_2)m(dU_1dA_1dU_2dA_2)
\end{aligned}
\end{align*}
for each $m\in\Pi(P_1,P_2)$, so that
\[
\left|\int_E\phi(V,B)P_1(dVdB))-\int_E\phi(V,B)P_2(dVdB)\right|\le\Lip(\phi)\MK(P_1,P_2).
\]
It is well known that the metric space $(\cP_1(E),\MK)$ is separable and complete (see Proposition 7.1.5 in \cite{AmbrosioGigliSavare}), and that $\MK$ metrizes the topology of weak convergence of Borel probability measures on
$E$ with some additional tightness condition at infinity: see Theorem 7.12 in \cite{VillaniAMS}.

For each $T>0$, equip $\cE_T$ with the metric $d_T$ of uniform convergence: in other words, for all $f_1,f_2\in\cE_T$, set
\[
d_T(f_1,f_2):=\sup_{0\le t\le T}\MK(f_1(t),f_2(t))\,.
\]
It is well known that $(E_T,d_T)$ is also a complete metric space. For $f^0\in\cP_1(E)$, set
\[
\cF_T:=\{g\in\cE_T\hbox{ s.t. }g(0)=f^0\}.
\]

\begin{proof}[Proof of Proposition \ref{P-LoheKinEq}]
Let $T>0$ (to be chosen later), and consider the map $\cT:\,\cF_T\to\cF_T$ defined by the following prescription: for each $g\in\cF_T$, the function $f=\cT g$ is the solution of the Cauchy problem
\[
\left\{\begin{aligned}
{}&\d_tf(t)+\Div_U\left(f(t)(AU+\tfrac\kappa{2}(\lA Vg(t)\rA-U\lA V^*g(t)\rA U))\right)=0,
\\
&f(0)=f^0,
\end{aligned}
\right.
\]
where
\[
\lA\phi m\rA=\int_E\phi(U,A)m(dUdA)
\]
for each $\phi\in C_b(E)$ and each Borel measure $m$ on $E$ with finite total variation (not necessarily positive). The method of characteristics implies that
\[
f(t)=\Phi(t,0)\#f^0,
\]
where $\Phi(t,s)(V,B)$ is the unique solution of the Cauchy problem
\[
\frac{d}{dt}(U,A)=(AU+\tfrac\kappa{2}(\lA Vg(t)\rA-U\lA V^*g(t)\rA U),0)\,,\quad (U,A)(s)=(V,B).
\]
See Proposition 8.1.8 in \cite{AmbrosioGigliSavare}. The existence of $\Phi$ once $g\in\cF_T$ is given follows directly from the Cauchy-Lipschitz theorem (notice indeed that the first component of the vector field above is quadratic in 
$U\in\bU(d)$ which is compact, and linear in $A$, and therefore is globally Lipschitz continuous on $E$. Pick another element of $\cF_T$, denoted $\bar g$, and denote by $\bar\Phi(t,s)$ the corresponding characteristic flow. We seek 
to bound 
\[
d_T(f,\bar f)\hbox{Ê in terms of }d_T(g,\bar g),\quad\hbox{ where }f=\cT g\hbox{Ê and }\bar f=\cT\bar g.
\]

Observe that
\[
\Phi(t,0)(U^0,A)-\bar\Phi(t,0)(V^0,B)=:(U(t)-V(t),A-B)
\]
and
\[
\begin{aligned}
&\dot U(t)-\dot V(t)= AU(t)-BV(t)+\tfrac\kappa{2}\lA W(g(t)-\bar g(t))\rA
\\
& \hspace{0.2cm} -\tfrac\kappa{2}U(t)\lA W^*g(t)\rA U(t)+\tfrac\kappa{2}V(t)\lA W^*g(t)\rA V(t) -\tfrac\kappa{2}V(t)\lA W^*(g(t)-\bar g(t))\rA V(t).
\end{aligned}
\]
Thus
\[
\begin{aligned}
\frac{d}{dt}\|U-V\|_2^2=&\Tr((U-V)^*(AU-BV)+(U^*A^*-V^*B^*)(U-V)
\\
&+\tfrac\kappa{2}\Tr((U-V)^*\lA W^*(g-\bar g)\rA+\lA W^*(g-\bar g)\rA(U-V))
\\
&-\tfrac\kappa{2}\Tr((U-V)^*(U\lA W^*g\rA U-V\lA W^*g\rA V))
\\
&-\tfrac\kappa{2}\Tr((U^*\lA Wg\rA U^*-V^*\lA Wg\rA V^*)(U-V))
\\
&-\tfrac\kappa{2}\Tr((U-V)^*V\lA W^*(g-\bar g)\rA V)
\\
&-\tfrac\kappa{2}\Tr(V^*\lA W(g-\bar g)\rA V^*(U-V))
\\
=&\mathcal{I}_{21}+\mathcal{I}_{22}+\mathcal{I}_{23}+\mathcal{I}_{24}+\mathcal{I}_{25}.
\end{aligned}
\]
The term $\mathcal{I}_{21}$ is handled by \eqref{TrAU}:
\[
\mathcal{I}_{21}\le 2\|A-B\|_2\|U-V\|_2,
\]
while the terms $\mathcal{I}_{23}$ and $\mathcal{I}_{24}$ are handled by Lemma \ref{L-AuxComp}:
\[
\begin{aligned}
\mathcal{I}_{23}+\mathcal{I}_{24}\le-2\ka\|U-V\|_2^2&+\tfrac\kappa{2}\Tr(\lA(W-V)(W-V)^*g\rA(U-V)(U-V)^*)
\\
&+\tfrac\kappa{2}\Tr(\lA(W-U)(W-U)^*g\rA(U-V)(U-V)^*)
\\
\le-2\ka\|U-V\|_2^2&+4\kappa\|U-V\|_2^2=2\ka\|U-V\|_2^2\,.
\end{aligned}
\]
The last inequality uses the identity and the second inequality in \eqref{IneqNorms} to prove that
\begin{align*}
\begin{aligned}
& \Tr(\lA(W-U)(W-U)^*g\rA(U-V)(U-V)^*)
\\
& =\|(U-V)^*\lA(W-U)(W-U)^*g\rA^{1/2}\|_2^2 \le\|\lA(W-U)(W-U)^*g\rA\|\|U-V\|_2^2.
\end{aligned}
\end{align*}
On the other hand, $\|M\|=1$ for each $M\in\bU(d)$, so that
\[
\|\lA(W-V)(W-V)^*g\rA\|\le\lA\|(W-V)(W-V)^*\|g\rA\le 4.
\]

It remains to treat the terms $\mathcal{I}_{22}$ and $\mathcal{I}_{25}$. Observe that
\begin{align*}
\begin{aligned}
\mathcal{I}_{22}&\le \ka\|U-V\|_2\|\lA W^*(g(t)-\bar g(t))\rA\|_2 \le \ka\|U-V\|_2\lA\|W^*\|_2(g(t)-\bar g(t))\rA
\\
&\le \ka\|U-V\|_2\MK(g(t),\bar g(t)),
\end{aligned}
\end{align*}
since $(U,A)\mapsto\|U\|_2$ is Lipschitz continuous with Lipschitz constant $1$. By the same token
\begin{align*}
\begin{aligned}
\mathcal{I}_{25}&=-\tfrac\kappa{2}\Tr((U-V)^*V\lA W^*(g(t)-\bar g(t))\rA V) \\
&-\tfrac\kappa{2}\Tr(V^*\lA W(g(t)-\bar g(t))\rA V^*(U-V))
\\
&\le \kappa\|U-V\|_2\|V\|\|\lA W^*(g(t)-\bar g(t))\rA V\|_2\|V\|
\\
&\le \kappa\|U-V\|_2\|V\|\MK(g(t),\bar g(t))\|V\|
\\
&\le \kappa\|U-V\|_2\MK(g(t),\bar g(t)).
\end{aligned}
\end{align*}
Therefore
\begin{align*}
\begin{aligned}
\frac{d}{dt}\|U-V\|_2^2 &\le2\|A-B\|_2\|U-V\|_2+\kappa\|U-V\|_2\MK(g(t),\bar g(t))
\\
&+2\ka\|U-V\|_2^2+\kappa\|U-V\|_2\MK(g(t),\bar g(t)),
\end{aligned}
\end{align*}
so that
\begin{align*}
\begin{aligned}
&\frac{d}{dt}(\|U-V\|_2+\|A-B\|_2) \\
& \hspace{0.5cm} =\frac{d}{dt}\|U-V\|_2 \le\|A-B\|_2+\ka\|U-V\|_2+\ka\MK(g(t),\bar g(t)).
\end{aligned}
\end{align*}
Thus
\begin{align*}
\begin{aligned}
& (\|U-V\|_2+\|A-B\|_2)(t)\le(\|U-V\|_2+\|A-B\|_2)(0)
\\
& \hspace{0.5cm} \le\max(1,\ka)\int_0^t(\|U-V\|_2+\|A-B\|_2)(s)ds+\ka\int_0^t\MK(g(s),\bar g(s))ds,
\end{aligned}
\end{align*}
or equivalently
\begin{align*}
\begin{aligned}
& |||\Phi(t,0)(U^0,A^0)-\bar\Phi(t,0)(V^0,B^0)||| \\
& \hspace{0.2cm} \le(\|U^0-V^0\|_2+\|A^0-B^0\|_2) \\
& \hspace{0.2cm} +\max(1,\ka)\int_0^t|||\Phi(s,0)(U^0,A^0)-\bar\Phi(s,0)(V^0,B^0)|||ds
 +\ka\int_0^t\MK(g(s),\bar g(s))ds,
\end{aligned}
\end{align*}
with the notation
\[
|||(U,A)-(V,B)|||:=\|U-V\|_2+\|A-B\|_2\,,\qquad U,V,A,B\in M_d(\bbc).
\]

Pick $m^0\in\Pi(f^0,f^0)$ and integrate both sides of the inequality above with respect to $m^0$: one finds that
\[
\begin{aligned}
\cD(t):=&\int_{E\times E}|||\Phi(t,0)(U^0,A^0)-\bar\Phi(t,0)(V^0,B^0)|||m^0(dU^0dA^0dV^0dB^0)
\\
\le&\cD(0)+\max(1,\ka)\int_0^t D(s)ds+\ka\int_0^t\MK(g(s),\bar g(s))ds.
\end{aligned}
\]
By Gronwall's lemma
\[
\cD(t)\le\left(\cD(0)\ka\int_0^t\MK(g(s),\bar g(s))ds\right)e^{\max(1,\ka)t}.
\]
By definition 
\[
D(t)\ge\MK(\Phi(t,0)\#f^0,\bar\Phi(t,0)\#f^0)=\MK(f(t),\bar f(t))
\]
for all $m^0\in\Pi(f^0,f^0)$, while, choosing 
\[
m^0(dUdAdVdB):=f^0(dUdA)\de_U(dV)\de_A(dB)
\]
implies that $D(0)=0$. Hence
\[
\MK(f(t),\bar f(t))\le\ka e^{\max(1,\ka)t}\int_0^t\MK(g(s),\bar g(s))ds,
\]
so that
\[
d_T(f,\bar f)\le \ka Te^{\max(1,\ka)T}d_T(g,\bar g).
\]
Therefore, if $\ka T\exp(\max(1,\ka)T)<1$, the map $\cT$ is a strict contraction on the closed subset $\cF_T$ of the complete metric space $(\cE_T,d_T)$. Therefore, $\cT$ has a unique fixed point in $\cF_T$ for each $f^0\in\cP_1(E)$, which
we call $f_1$. This fixed point $f_1$ is a weak solution of the kinetic Lohe equation by construction (according to the method of characteristics).

Define by induction $f_n$ to be the unique solution of the kinetic Lohe equation on $[0,T]$ with initial data $f_{n+1}\rstr_{t=0}=f_n(T)$. Set 
\[
f(t)=f_n(t-nT)\qquad\hbox{Êfor all }t\in[nT,(n+1)T)\hbox{ and all }n\ge 0.
\]
Then $f$ is the unique weak solution of the kinetic Lohe equation with initial data $f^0$ in the space $C([0,+\infty),(\cP_1(E),\MK))$.

Finally, the method of characteristics implies that
\[
f(t)=\Psi(t,0)\#f^0,
\]
where $\Psi(t,s)(V,B)$ is the unique solution of the Cauchy problem
\[
\frac{d}{dt}(U,A)=(AU+\tfrac\kappa{2}(\lA V^*f(t)\rA-U\lA V^*,f(t)\rA U),0)\,,\quad (U,A)(s)=(V,B).
\]
Since $\Psi(t,0)(V,B)$ is of the form $\Psi(t,0)(V,B)=(U(t),B)$ for all $t\ge 0$, we conclude that
\[
\int_E\|A\|_2^pf(t,dUdA)=\int_E\|A\|_2^pf^0(dUdA)
\]
for each $p\ge 1$. In particular
\[
f^0\in\cP_p(E)\implies f(t)\in\cP_p(E)\quad\hbox{Êfor all }t\ge 0.
\]
\end{proof}


\section{A Barrier Function}\label{App:B}


Let $\eta\in[0,(2/3)^{3/2})$. The equation
\[
\tfrac12z^3-z+\eta=0
\]
has two nonnegative roots $\zeta_1(\eta)$ and $\zeta_2(\eta)$ such that
\[
0\le\zeta_1(\eta)<\sqrt{2/3}<\zeta_2(\eta)\le\sqrt2.
\]
For all $z\ge 0$, one has
\[
\tfrac12z^3-z+\eta<0\hbox{ if and only if }\zeta_1(\eta)<z<\zeta_2(\eta).
\]
Notice that
\[
\zeta_1(0)=0\quad\hbox{ and }\quad\zeta_2(0)=\sqrt{2}.
\]

\begin{Lem}\label{L-Barrier}
Assume that $\ka>(3/2)^{3/2}\a\ge 0$, and let $0<\Dlt^0<\zeta_2(\a/\ka)$. On the other hand, let $x\in C^1([0,+\infty)$ satisfy $0\le x(t)\le 2\sqrt{d}$ for all $t\ge 0$ and
\[
\dot x(t)\le\a-\ka x(t)+\tfrac12\ka \Dlt(t)^2x(t),\quad x(0)=x^0,\quad \Dlt(t):=\sup_{x(0)\in\Si}x(t),
\]
where
\[
\Si\subset[0,+\infty)\quad\hbox{ and }\Dlt^0:=\sup(\Si).
\]
Then 
\[
\Dlt(t)\le y(t)\quad\hbox{ for each }t\ge 0,
\]
where $y$ is the solution of the Cauchy problem
\[
\dot y=\a-\ka y+\tfrac12\ka y^3,\quad y(0)=\Dlt^0.
\]
\end{Lem}

\begin{proof}
Pick $0<\eps<\zeta_2(\a/\ka)-\Dlt^0$ and let $z_\eps$ be the solution of
\[
\dot z=\a-\ka z+\tfrac12\ka z^3,\quad z(0)=\Dlt^0+\eps.
\]
One easily check that this Cauchy problem has indeed a unique solution $z_\eps$ defined on $[0,+\infty)$ which satisfies 
\[
0\le z_\eps(t)\le\zeta_2(\a/\ka)\quad\hbox{ for all }t\ge 0.
\]

Set
\[
S:=\{t\ge 0\hbox{ s.t. }x(s)\le z_\eps(s)\hbox{ for all }s\in[0,t]\hbox{Ê and all }x^0\in[0,\Dlt^0]\}.
\]
Obviously $0\in S$ so that $S\not=\varnothing$.

Let $0<t_1<t_2<\ldots<t_n\in S$ satisfy $t_n\to T$. Then $x(t)\le z_\eps(t)$ for all $t\in[0,T)$ and all $x^0\in[0,\Dlt^0]$. By the continuity of $x$ and $z_\eps$, one has $x(T)\le z_\eps(T)$, so that $T\in S$. In other words, $S$ is closed.

Let $t\in S$; then
\[
\dot x(s)\le\a-\ka x(s)+\tfrac12\ka z_\eps(s)^2x(s),\quad x(0)=x^0\,,\quad 0\le s<t
\]
so that
\begin{align*}
\begin{aligned}
x(t) &\le x^0\exp\left(-\ka\int_0^t(1-\tfrac12z_\eps(s)^2)ds\right) +\a\int_0^t\exp\left(-\ka\int_\tau^t(1-\tfrac12z_\eps(s)^2)ds\right)d\tau
\\
&\le \Dlt^0\exp\left(-\ka\int_0^t(1-\tfrac12z_\eps(s)^2)ds\right) +\a\int_0^t\exp\left(-\ka\int_\tau^t(1-\tfrac12z_\eps(s)^2)ds\right)d\tau=:\Dlt_t^\eps
\\
& <(\Dlt^0+\eps)\exp\left(-\ka\int_0^t(1-\tfrac12z_\eps(s)^2)ds\right) +\a\int_0^t\exp\left(-\ka\int_\tau^t(1-\tfrac12z_Öeps(s)^2)ds\right)d\tau \\
& =z(t).
\end{aligned}
\end{align*}
Since $0\le x(t)\le 2\sqrt{d}$ for each $t\ge 0$, one has 
\[
0\le\Dlt(t)\le 2\sqrt{d}
\]
for all $t\ge 0$, so that
\[
\dot x(s)\le\a-\ka x(s)+2d\ka x(s),\quad s>t.
\]
Hence
\[
\begin{aligned}
x(s)\le&x(t)e^{(2d-1)\ka(s-t)}+\a\frac{e^{(2d-1)\ka(s-t)}-1}{(2d-1)\ka}
\\
\le&\Dlt_te^{(2d-1)\ka(s-t)}+\a\frac{e^{(2d-1)\ka(s-t)}-1}{(2d-1)\ka}=:\Dlt_s^\eps,\qquad s>t.
\end{aligned}
\]
Since $s\mapsto \Dlt_s^\eps$ is continuous on $[t,+\infty)$ and $\Dlt_t^\eps<z(t)$, there exists $\eta>0$ such that $\Dlt_s^\eps<z(s)$ for each $s\in[t,t+\eta)$. Hence
\[
x(s)\le z_\eps(s)\hbox{ for all }s\in[t,t+\eta)\hbox{ and all }x^0\in\Si\subset[0,\Dlt^0].
\]
Therefore $S$ is open. 

Since $S$ is a nonempty subset of $[0,+\infty)$ that is both closed and open, one has $S=[0,+\infty)$. Thus
\[
x(t)\le z_\eps(t)\hbox{ for all }t\ge 0\hbox{Ê and all }x^0\in\Si\subset[0,\Dlt^0].
\]
Letting $\eps\to 0$ and observing that $z_\eps(t)\to y(t)$ for each $t\ge 0$ as $\eps\to 0$, we conclude that
\[
x(t)\le y(t)\hbox{ for all }t\ge 0,\hbox{ all }\eps\in(0,\zeta_2(\a/\ka)-\Dlt^0),\hbox{Ê and all }x^0\in\Si\subset[0,\Dlt^0],
\]
which proves our claim.
\end{proof}

\end{appendix}



\end{document}